\documentclass[12pt]{article}

\usepackage{graphicx}
\usepackage{multicol,multirow}
\usepackage{ytableau}
\usepackage{genyoungtabtikz}
\usepackage{amsmath,amssymb,amsfonts}
\usepackage{amsthm}%
\usepackage{mathrsfs}
\usepackage{rotating}
\usepackage{appendix}
\usepackage[numbers]{natbib}
\usepackage{tikz}
\usetikzlibrary{positioning}
\usetikzlibrary {arrows.meta,bending}

\allowdisplaybreaks

\usepackage{xcolor}%
\newcommand{\twelve}{12}

{\theoremstyle{plain}%
 \newtheorem{theorem}{Theorem}

}
{\theoremstyle{remark}
\newtheorem{remark}{Remark}
}
{\theoremstyle{definition}
\newtheorem{definition}{Definition}
\newtheorem{example}{Example}
}

\begin{document}

\begin{center}
{\Large Partition algebras as monoid algebras}

 \ 

{\textsc{John M. Campbell}} 

 \ 

\end{center}

\begin{abstract}
Wilcox has considered a \emph{twisted} semigroup algebra structure on the partition algebra $\mathbb{C}A_k(n)$, but it 
 appears that there has not previously been any known basis that gives $\mathbb{C}A_k(n)$ the structure of a ``non-twisted'' semigroup 
 algebra or a monoid algebra. This motivates the following problem, for the non-degenerate case whereby $n \in \mathbb{C} 
 \setminus \{ 0, 1, \ldots, 2 k - 2 \}$ so that $ \mathbb{C}A_k(n)$ is semisimple. How could a basis $M_{k} = M$ of $ \mathbb{C}A_k(n)$ 
 be constructed so that $M$ is closed under the multiplicative operation on $\mathbb{C}A_k(n)$, in such a way so that $M$ is a 
 monoid under this operation, and how could a product rule for elements in $M$ be defined in an explicit and combinatorial way in 
 terms of partition diagrams? We construct a basis $M$ of the desired form using Halverson and Ram's matrix unit construction for 
 partition algebras, Benkart and Halverson's bijection between vacillating tableaux and set-partition tableaux, an analogue given by 
 Colmenarejo et al.\ for partition diagrams of the RSK correspondence, and a variant of a result due to Hewitt and Zuckerman 
 characterizing finite-dimensional semisimple algebras that are isomorphic to semigroup algebras.
\end{abstract}

\vspace{0.1in}

\noindent {{\footnotesize \emph{MSC:} 05E10, 20M25}}

\vspace{0.1in}

\noindent {{\footnotesize \emph{Keywords:}  partition algebra, semigroup algebra, twisted semigroup algebra, monoid
  algebra, partition diagram, Schur--Weyl duality, vacillating tableau, set-partition tableau, semisimple algebra, RSK 
 correspondence, matrix unit}}

\section{Introduction}\label{sectionIntro}
 \emph{Diagram algebras} provide an important area of study within algebraic combinatorics, with direct applications in statistical 
 mechanics, knot theory, and invariant theory. Diagram algebras broadly refer to subalgebras of the unital, associative algebra known as 
 the \emph{partition algebra} $ \mathbb{C}A_k(n)$, which emerged through the work of Martin 
 \cite{Martin1991,Martin1994,Martin1996,Martin2000} and Jones \cite{Jones1994}. Partition algebras may be seen as arising through a 
 counterpart to classical versions of Schur--Weyl duality relating the irreducible representations of the general linear group 
 $\text{GL}_{n}(\mathbb{C}) $ and the symmetric group $S_{k}$ via the actions of these groups on the tensor power $V^{\otimes k}$ for 
 an $n$-dimensional complex vector space $V$. As suggested by East \cite{East2011}, the complicated nature of the multiplicative 
 operation on $\mathbb{C}A_k(n)$, according to the usual bases of $\mathbb{C}A_k(n)$, motivates the development of simplified ways 
 of expressing or constructing or defining $\mathbb{C}A_k(n)$, as in with the use of presentations via generators and relations. In this 
 paper, we introduce a basis of $\mathbb{C}A_k(n)$ with a multiplication rule that may be thought of as being simpler and more ideal 
 relative to previously known bases of $\mathbb{C}A_k(n)$. Notably, and in contrast to previously known bases of $\mathbb{C}A_k(n)$, 
 our new basis is closed under the product operation on $\mathbb{C}A_k(n)$ and gives $\mathbb{C}A_k(n)$ the structure of a monoid 
 algebra (for the non-degenerate cases whereby $\mathbb{C}A_k(n)$ is semisimple). 

 Wilcox \cite{Wilcox2007}, in 2007, showed how diagram algebras are \emph{twisted semigroup algebras}, building on the work of Clark 
 \cite{Clark1967}. The two distinguished bases of $\mathbb{C}A_k(n)$ are the \emph{diagram basis} $\{ d_{\pi} \}_{\pi}$ and the 
 \emph{orbit basis} $\{ o_{\pi} \}_{\pi}$, and the multiplication rules for these canonical bases can be thought of as giving rise to the 
 twisted semigroup algebra structure on $ \mathbb{C}A_k(n)$. The $n = 1$ case gives us that $\mathbb{C}A_k(n)$ is equal to the 
 monoid algebra $\mathbb{C}A_k$ for the partition monoid $A_k$, but it seems that $\mathbb{C}A_k(n)$ has not otherwise been 
 considered as a ``non-twisted'' semigroup algebra or as a monoid algebra. Moreover, the twisted semigroup algebra structure on 
 $ \mathbb{C}A_k(n)$ does not give rise, in any direct way, to a semigroup algebra structure on $\mathbb{C}A_k(n)$, since, as it can be 
 shown, for any combination of nonzero scalars of the form $s_{\pi}$ for order-$k$ partition diagrams $\pi$, the families $\{ s_{\pi} d_{\pi} 
 \}_{\pi}$ and $\{ s_{\pi} o_{\pi} \}_{\pi}$ are not bases that are closed under the multiplicative operation on $\mathbb{C}A_k(n)$. 

 A result of basic importance in the representation theory of 
 partition algebras is such that if $n \in \mathbb{C} \setminus \{ 0, 1, 
 \ldots, 2 k - 2 \}$, then 
 $\mathbb{C}A_k(n)$ is semisimple. So, with the understanding that $n$ is not among the finite number of 
 degenerate cases such that $\mathbb{C}A_k(n)$ is not semisimple, the parameter $n$ can, informally, be ``disregarded'' in the sense 
 that it does not provide any information about $\mathbb{C}A_k(n)$ or its semisimple structure. The foregoing considerations lead us 
 toward the following problem that is central to our work. 

 \vspace{0.1in}

\noindent \emph{Problem:} Construct a basis $M_k = M$ of $\mathbb{C}A_k(n)$, for the non-degenerate case such that $n \in 
 \mathbb{C} \setminus \{ 0, 1, \ldots, 2 k - 2 \}$, such that 

 \vspace{0.1in}

\noindent $\bullet$ The basis $M$ is a monoid under the product operation on $\mathbb{C}A_k(n)$; and 

 \vspace{0.1in}
 
\noindent $\bullet$ The basis $M$ is indexed by (order-$k$) partition diagrams in such a way so that the product $m_{\alpha} 
 m_{\beta} \in M$ may be evaluated in an explicit and combinatorial way, i.e., as an element $m_{\gamma} \in M$ for a partition diagram 
 $\gamma$ obtained through the application of a combinatorial rule to the partition diagrams $\alpha$ and $\beta$. 

 \vspace{0.1in}

 We succeed in solving this problem. The main ingredients in our solution are given by 

 \vspace{0.1in}

\noindent $\bullet$ A matrix unit construction due to Halverson and Ram \cite{HalversonRam2005} 
 obtained through a formulation of the basic construction due to Bourbaki 
 \cite[\S2]{Bourbaki1990}; 

 \vspace{0.1in}

 \noindent $\bullet$ A bijection between vacillating tableaux and set-partition tableaux
 due to Benkart and Halverson \cite{BenkartHalverson2019trends}; 

 \vspace{0.1in}

 \noindent $\bullet$ An analogue 
 given by Colmenarejo et al.\ \cite{ColmenarejoOrellanaSaliolaSchillingZabrocki2020} 
 for partition diagrams of the RSK correspondence; and 

 \vspace{0.1in}

\noindent $\bullet$ A variant of a theorem due to Hewitt and Zuckerman \cite{HewittZuckerman1955} 
 giving a necessary and sufficient 
 condition for a finite-dimensional semisimple algebra over an algebraically closed field 
 to be isomorphic to a semigroup algebra. 

 \vspace{0.1in}

 Our construction of a monoid algebra from matrix units follows a similar approach as in our recent work on semisimple algebras and 
 immaculate tableaux \cite{Campbellunpublished}, but this past approach did not concern partition algebras or partition diagrams. 
 For partition diagrams $\alpha$ and $\beta$, 
 the product of elements $d_{\alpha}$ and $d_{\beta}$ in the diagram basis of $\mathbb{C}A_k(n)$ satisfies 
 \begin{equation*}
 d_{\alpha} d_{\beta} = n^{\ell(\alpha, \beta)} d_{\alpha \circ \beta} 
\end{equation*}
 for a statistic $\ell(\alpha, \beta)$ depending on $\alpha$ and $\beta$ and for a binary operation $\circ$ on partition diagrams that 
 we later review. The product of elements $o_{\alpha}$ and $o_{\beta}$ either vanishes or is equal to a nonvanishing linear 
 combination of orbit basis elements with $n$-polynomial coefficients. This illustrates how our new basis $M = M_{k} = \{ 
 m_{\pi} \}_{\pi}$ of $\mathbb{C}A_k(n)$ may be seen as more natural or more ideal compared with both $\{ d_{\pi} \}_{\pi}$ and $\{ 
 o_{\pi} \}_{\pi}$, if we consider the closure property 
\begin{equation}\label{malphambeta}
 m_{\alpha} m_{\beta} = m_{\gamma} 
\end{equation}
 relative to the diagram and orbit bases. 

\subsection{A motivating result}\label{subsectionmotivating}
 As a way of demonstrating the effectiveness of our construction and of clarifying how the parameter $n$ involved in the definition of 
 $\mathbb{C}A_k(n)$ can be ``disregarded'' in the sense described above, we provide a concrete and explicit illustration of our 
 construction, as below, referring to Section \ref{sectionbackground} for a full definition for $\mathbb{C}A_k(n)$. 

 Our main construction gives an explicitly defined basis $M = M_{k} = \{ m_{\pi} \}_{\pi}$ of $\mathbb{C}A_k(n)$ (again for the 
 non-degenerate cases) of the desired form described above. However, our construction can be modified to give rise to infinitely 
 many bases for $\mathbb{C}A_k(n)$ satisfying the given conditions, and this is described later in this paper. 
 For certain technical reasons, and for the purposes of the following illustration, 
 it is more convenient to use one of our modified bases, denoted here with 
 $\widetilde{M} = \widetilde{M}_{k} = \{ \widetilde{m}_{\pi} \}_{\pi}$. 

\begin{figure}
\begin{center}
{\tiny{
\begin{equation*}
\rotatebox{0}{ $ \left( \begin{tabular}{ c@{\hskip 0.006in}c@{\hskip 0.006in}c@{\hskip 0.006in}c@{\hskip 0.006in}c@{\hskip 0.006in}c@{\hskip 0.006in}c@{\hskip 0.006in}c@{\hskip 0.006in}c@{\hskip 0.006in}c@{\hskip 0.006in}c@{\hskip 0.006in}c@{\hskip 0.006in}c@{\hskip 0.006in}c@{\hskip 0.006in}c}
 $0$ & $0$ & $0$ & $0$ & $0$ & $0$ & $\tfrac{-1}{2}$ & $\tfrac{1}{2n}$ & $ \tfrac{1}{2}$ & $0$ & $ \tfrac{1}{2n}$ & $ \tfrac{-1}{2n}$ & 
 $ \tfrac{1}{2n}$ & $0$ & $ \tfrac{-1}{n^2}$ \\ 
 $0$ & $0$ & $0$ & $0$ & $0$ & $0$ & $\tfrac{-1}{2}$ & $\tfrac{1}{2n}$ & $\tfrac{1}{2}$ & $1$ & $ \tfrac{-3}{2n}$ & $ \tfrac{-1}{2n}$ & 
 $ \tfrac{1}{2n}$ & $\tfrac{-1}{n}$ & $ \tfrac{1}{n^2} $ \\ 
 $0$ & $0$ & $0$ & $0$ & $0$ & $0$ & $\tfrac{-1}{2}$ & $\tfrac{1}{2n}$ & $ \tfrac{1}{2}$ & $-1$ & $ \tfrac{1}{2n} $ & 
 $ \tfrac{-1}{2n}$ & $ \tfrac{2n-1}{2n}$ & $ \tfrac{1}{n} $ & $ \tfrac{-1}{n} $ \\ 
 $0$ & $0$ & $0$ & $0$ & $0$ & $0$ & $\tfrac{-1}{2}$ & $ \tfrac{1}{2n} $ & $ \tfrac{1}{2} $ & $0$ & $ \tfrac{-1}{2n} $ & 
 $ \tfrac{-1}{2n} $ & $ \tfrac{1}{2n} $ & $0$ & $ \tfrac{1}{n^2} $ \\ 
 $0$ & $0$ & $0$ & $0$ & $0$ & $0$ & $ \tfrac{-1}{2} $ & $ \tfrac{1}{2n} $ & $ \tfrac{1}{2} $ & $ 0 $ & $ \tfrac{-1}{2n} $ & 
 $ \tfrac{-1}{2n} $ & $ \tfrac{1}{2n} $ & $ \tfrac{1}{n} $ & $ \tfrac{-1}{n^2} $ \\ 
 $0$ & $ \tfrac{-1}{\mathcal{D}(n)} $ & $ 0 $ & $ 0 $ & $ \tfrac{1}{ \mathcal{D}(n) n } $ & $ 0$ & $ \tfrac{-1}{2} $ & $ \tfrac{1}{2(n-2)} $ & 
 $ \tfrac{1}{2} $ & $0$ & $ \tfrac{3-n}{2 \mathcal{D}(n) } $ & $ \tfrac{-1}{2n} $ & $ \tfrac{1}{2n} $ & $0$ & $ \tfrac{-1}{ \mathcal{D}(n) n } $ \\ 
 $0$ & $0$ & $0$ & $0$ & $0$ & $0$ & $ \tfrac{-1}{2} $ & $ \tfrac{1}{2n} $ & $ \tfrac{1}{2} $ & $ 0 $ & $ \tfrac{-1}{2n} $ & 
 $ \tfrac{-1}{2n} $ & $ \tfrac{1}{2n} $ & $ 0$ & $ 0 $ \\ 
 $\tfrac{-n}{\mathcal{D}(n)}$ & $\tfrac{1}{\mathcal{D}(n)}$ & $0$ & $\tfrac{1}{\mathcal{D}(n)}$ & $\tfrac{-1}{\mathcal{D}(n)n}$ & 
 $\tfrac{1}{n-2}$ & $\tfrac{-1}{2}$ & $\tfrac{n-4}{2(n-2)n}$ & $\tfrac{1}{2}$ & $\tfrac{1}{\mathcal{D}(n)}$ & 
 $\tfrac{3 n -n^2 -4}{2\mathcal{D}(n) n}$ & $\tfrac{-1}{2n}$ & $\tfrac{1}{2n}$ & $\tfrac{-1}{\mathcal{D}(n)}$ & 
 $\tfrac{1}{\mathcal{D}(n)n}$ \\ 
 $0$ & $0$ & $0$ & $0$ & $0$ & $0$ & $0$ & $0$ & $1$ & $0$ & $0$ & $0$ & $0$ & $0$ & $0$ \\ 
 $0$ & $\frac{1}{n-1}$ & $0$ & $0$ & $\frac{-1}{(n-1)n}$ & $0$ & $\frac{-1}{2}$ & $\frac{1}{2n}$ & $\frac{1}{2}$ & $0$ & 
 $\frac{n+1}{2(1-n)n}$ & $\frac{-1}{2n}$ & $\frac{1}{2n}$ & $0$ & $\frac{1}{(n-1)n^2}$ \\ 
 $\frac{n}{n-1}$ & $\frac{-1}{n-1}$ & $0$ & $\frac{-1}{n-1}$ & $\frac{1}{(n-1)n}$ & $0$ & $\frac{-1}{2}$ & $\frac{1}{2n}$ & 
 $\frac{1}{2}$ & $\frac{-1}{n-1}$ & $\frac{3-n}{2(n-1) n}$ & $\frac{-1}{2n}$ & $\frac{1}{2n}$ & $\frac{1}{(n-1)n}$ & 
 $\frac{-1}{(n-1)n^2}$ \\ 
 $\frac{n}{\mathcal{D}(n)}$ & $\frac{-1}{\mathcal{D}(n)}$ & $\frac{-1}{n-2}$ & $\frac{-1}{\mathcal{D}(n)}$ & 
 $\frac{1}{\mathcal{D}(n)}$ & $\frac{-1}{n-2}$ & $\frac{-1}{2}$ & $\frac{1}{2(n-2)}$ & $\frac{1}{2}$ & 
 $\frac{-1}{\mathcal{D}(n)}$ & $\frac{3-n}{2\mathcal{D}(n)} $ & $\frac{1}{2(n-2)}$ & $\frac{1}{2(n-2)}$ & 
 $\frac{1}{\mathcal{D}(n)}$ & $\frac{-1}{\mathcal{D}(n)}$ \\ 
 $\frac{-n}{n-1}$ & $\frac{1}{n-1}$ & $1$ & $ \frac{1}{n-1} $ & $\frac{-1}{n-1}$ & $0$ & $\frac{-1}{2}$ & $\frac{1}{2n}$ & 
 $\frac{1}{2}$ & $\frac{1}{n-1}$ & $\frac{n+1}{2(1-n)n}$ & $\frac{-1}{2n}$ & $\frac{-1}{2n}$ & $\frac{-1}{(n-1)n}$ & 
 $\frac{1}{(n-1)n}$ \\ 
 $0$ & $0$ & $0$ & $0$ & $\frac{1}{(n-1)n}$ & $0$ & $\frac{-1}{2}$ & $\frac{1}{2n}$ & $\frac{1}{2}$ & $0$ & $\frac{-1}{2n}$ & 
 $\frac{-1}{2n}$ & $\frac{1}{2n}$ & $0$ & $\frac{-1}{(n-1)n^2}$ \\ 
 $0$ & $0$ & $0$ & $\frac{1}{n-1}$ & $\frac{-1}{(n-1)n}$ & $0$ & $\frac{-1}{2}$ & $\frac{1}{2n}$ & $\frac{1}{2}$ & $0$ & 
 $\frac{-1}{2n}$ & $\frac{-1}{2n}$ & $\frac{1}{2n}$ & $\frac{-1}{(n-1)n}$ & $\frac{1}{(n-1)n^2}$ 
 \end{tabular} \right) $ }
\end{equation*}}}
\caption{\label{figuretransition}
 The transition matrix $\widetilde{\mathcal{M}}_{2}$, 
 with the abbreviation $\mathcal{D}(n) = (n-2)(n-1)$} 
\end{center}
\end{figure}

 Let $n \in \mathbb{C} \setminus \{ 0, 1, 2 \}$. Again referring to Section \ref{sectionbackground} for explicit definitions, let the 
 diagram basis for the 15-dimensional partition algebra $\mathbb{C}A_2(n)$ be ordered by writing $d_{\pi_{1}} < d_{\pi_{2}} < \cdots < 
 d_{\pi_{15}}$, for the set $\{ \pi_{1}, \pi_{2}, \ldots, \pi_{15} \}$ of order-$2$ partition diagrams, and according to the ordering for 
 diagram bases implemented in the {\tt SageMath} Computer Algebra System. Explicitly, this gives us the linear ordering whereby 

 \ 

\noindent $d_{\{ \{ 2', 1', 1, 2 \} \}}$ $ < $ $d_{\{ \{2', 1, 2\}, \{1'\} \}}$ $ < $ $d_{\{ \{2'\}, \{1', 1, 2\} \}}$ $ < $ $d_{\{ \{2', 1'\}, \{1, 
 2\} \}}$ $ < $ $d_{\{ \{2'\}, \{1'\}, \{1, 2\} \}}$ $ < $ $d_{\{ \{2', 1', 1\}, \{2\} \}}$ $ < $ 
 $d_{\{ \{2', 1\}, \{1', 2\} \}}$ $ < $ 
 $d_{\{ \{2', 1\}, \{1'\}, \{2\} \}}$ $ < $ 
 $d_{\{ \{2', 2\}, \{1', 1\} \}}$ $ < $ 
 $d_{\{ \{2', 1', 2\}, \{1\} \}}$ $ < $ 
 $d_{\{ \{2', 2\}, \{1'\}, \{1\} \}}$ $ < $ 
 $d_{\{ \{2'\}, \{1', 1\}, \{2\} \}}$ $ < $ 
 $d_{\{ \{2'\}, \{1', 2\}, \{1\} \}}$ $ < $ 
 $d_{\{ \{2', 1'\}, \{1\}, \{2\} \}}$ $ < $ 

\noindent $d_{\{ \{2'\}, \{1'\}, \{1\}, \{2\} \}}. $ 

 \ 

\noindent We then define a family $\{ \widetilde{m}_{\pi_{i}} : i \in \{ 1, 2, \ldots, 15 \} \}$ of elements in $\mathbb{C}A_2(n)$ by writing 
 $$ \left[ \widetilde{m}_{\pi_{1}} \ \widetilde{m}_{\pi_{2}} \ \cdots \ \widetilde{m}_{\pi_{15}} \right]^{\text{T}} = 
 \widetilde{\mathcal{M}}_{2} \, \left[ d_{\pi_{1}} \ d_{\pi_{2}} \ \cdots \ d_{\pi_{15}} \right]^{\text{T}}, $$ where $ 
 \widetilde{\mathcal{M}}_{2} $ denotes the $15 \times 15$ transition matrix given explicitly in Figure \ref{figuretransition}. 

 We find that the determinant of $\widetilde{\mathcal{M}}_{2}$ is $ -\frac{1}{2 (n-2)^3 (n-1)^2 n^7}$, giving us that 
 $\{ \widetilde{m}_{\pi} \}_{\pi}$ is a basis of $\mathbb{C}A_2(n)$, writing $\widetilde{M}_{2} = \widetilde{M} = \{ 
 \widetilde{m}_{\pi} \}_{\pi}$. A remarkable property about this new basis $\widetilde{M}$ is given by how it is closed under the 
 product operation on $\mathbb{C}A_2(n)$ (in a way that does not depend on the parameter $n$, in contrast to the diagram and orbit 
 bases) and has the structure of a monoid under this multiplicative operation, so that $\mathbb{C}A_2(n)$ is equal to the monoid algebra 
 $\mathbb{C}\widetilde{M}$. All possible evaluations of (twofold) products of elements in $\widetilde{M}$ are given in 
 Figure \ref{Figuremulttable}. 

\begin{figure}
\begin{center}
{\footnotesize
\noindent \begin{tabular}{ c | c@{\hskip 0.006in}c@{\hskip 0.006in}c@{\hskip 0.006in}c@{\hskip 0.006in}c@{\hskip 0.006in}c@{\hskip 0.006in}c@{\hskip 0.006in}c@{\hskip 0.006in}c@{\hskip 0.006in}c@{\hskip 0.006in}c@{\hskip 0.006in}c@{\hskip 0.006in}c@{\hskip 0.006in}c@{\hskip 0.006in}c}
 \null & $\text{\color{black}$\widetilde{m}_{\pi_{1}}$}$
 & $\text{\color{black}$\widetilde{m}_{\pi_{2}}$}$ 
 & $\text{\color{black}$\widetilde{m}_{\pi_{3}}$}$ 
 & $\text{\color{black}$\widetilde{m}_{\pi_{4}}$}$
 & $\text{\color{black}$\widetilde{m}_{\pi_{5}}$}$ 
 & $\text{\color{black}$\widetilde{m}_{\pi_{6}}$}$
 & $\text{\color{black}$\widetilde{m}_{\pi_{7}}$}$ 
 & $\text{\color{black}$\widetilde{m}_{\pi_{8}}$}$ 
 & $\text{\color{black}id}$ 
 & $\text{\color{black}$\widetilde{m}_{\pi_{10}}$}$ 
 & $\text{\color{black}$\widetilde{m}_{\pi_{11}}$}$ 
 & $\text{\color{black}$\widetilde{m}_{\pi_{12}}$}$ 
 & $\text{\color{black}$\widetilde{m}_{\pi_{13}}$}$ 
 & $\text{\color{black}$\widetilde{m}_{\pi_{14}}$}$ 
 & $\text{\color{black}$\widetilde{m}_{\pi_{15}}$}$ \\ \hline
 $\text{\color{black}$\widetilde{m}_{\pi_{1}}$}$ 
 & $\text{\color{black}$\widetilde{m}_{\pi_{1}}$}$ 
 & $\text{\color{black}$\widetilde{m}_{\pi_{2}}$}$ 
 & $\text{\color{black}$\widetilde{m}_{\pi_{3}}$}$ 
 & $\text{\color{black}$\widetilde{m}_{\pi_{7}}$}$ 
 & $\text{\color{black}$\widetilde{m}_{\pi_{7}}$}$ 
 & $\text{\color{black}$\widetilde{m}_{\pi_{7}}$}$ 
 & $\text{\color{black}$\widetilde{m}_{\pi_{7}}$}$ 
 & $\text{\color{black}$\widetilde{m}_{\pi_{7}}$}$ 
 & $\text{\color{black}$\widetilde{m}_{\pi_{1}}$}$ 
 & $\text{\color{black}$\widetilde{m}_{\pi_{7}}$}$ 
 & $\text{\color{black}$\widetilde{m}_{\pi_{7}}$}$ 
 & $\text{\color{black}$\widetilde{m}_{\pi_{7}}$}$ 
 & $\text{\color{black}$\widetilde{m}_{\pi_{7}}$}$ 
 & $\text{\color{black}$\widetilde{m}_{\pi_{7}}$}$ 
 & $\text{\color{black}$\widetilde{m}_{\pi_{7}}$}$ \\ 
 $\text{\color{black}$\widetilde{m}_{\pi_{2}}$}$ 
 & $\text{\color{black}$\widetilde{m}_{\pi_{7}}$}$ 
 & $\text{\color{black}$\widetilde{m}_{\pi_{7}}$}$ 
 & $\text{\color{black}$\widetilde{m}_{\pi_{7}}$}$ 
 & $\text{\color{black}$\widetilde{m}_{\pi_{7}}$}$ 
 & $\text{\color{black}$\widetilde{m}_{\pi_{7}}$}$ 
 & $\text{\color{black}$\widetilde{m}_{\pi_{7}}$}$ 
 & $\text{\color{black}$\widetilde{m}_{\pi_{7}}$}$ 
 & $\text{\color{black}$\widetilde{m}_{\pi_{7}}$}$
 & $\text{\color{black}$\widetilde{m}_{\pi_{2}}$}$ 
 & $\text{\color{black}$\widetilde{m}_{\pi_{1}}$}$ 
 & $\text{\color{black}$\widetilde{m}_{\pi_{2}}$}$ 
 & $\text{\color{black}$\widetilde{m}_{\pi_{7}}$}$ 
 & $\text{\color{black}$\widetilde{m}_{\pi_{3}}$}$ 
 & $\text{\color{black}$\widetilde{m}_{\pi_{7}}$}$ 
 & $\text{\color{black}$\widetilde{m}_{\pi_{7}}$}$ \\
 $\text{\color{black}$\widetilde{m}_{\pi_{3}}$}$ 
 & $\text{\color{black}$\widetilde{m}_{\pi_{7}}$}$ 
 & $\text{\color{black}$\widetilde{m}_{\pi_{7}}$}$ 
 & $\text{\color{black}$\widetilde{m}_{\pi_{7}}$}$ 
 & $\text{\color{black}$\widetilde{m}_{\pi_{7}}$}$ 
 & $\text{\color{black}$\widetilde{m}_{\pi_{7}}$}$ 
 & $\text{\color{black}$\widetilde{m}_{\pi_{1}}$}$ 
 & $\text{\color{black}$\widetilde{m}_{\pi_{7}}$}$ 
 & $\text{\color{black}$\widetilde{m}_{\pi_{2}}$}$
 & $\text{\color{black}$\widetilde{m}_{\pi_{3}}$}$ 
 & $\text{\color{black}$\widetilde{m}_{\pi_{7}}$}$ 
 & $\text{\color{black}$\widetilde{m}_{\pi_{7}}$}$ 
 & $\text{\color{black}$\widetilde{m}_{\pi_{3}}$}$ 
 & $\text{\color{black}$\widetilde{m}_{\pi_{7}}$}$ 
 & $\text{\color{black}$\widetilde{m}_{\pi_{7}}$}$
 & $\text{\color{black}$\widetilde{m}_{\pi_{7}}$}$ \\ 
 $\text{\color{black}$\widetilde{m}_{\pi_{4}}$}$ 
 & $\text{\color{black}$\widetilde{m}_{\pi_{7}}$}$ 
 & $\text{\color{black}$\widetilde{m}_{\pi_{7}}$}$ 
 & $\text{\color{black}$\widetilde{m}_{\pi_{7}}$}$ 
 & $\text{\color{black}$\widetilde{m}_{\pi_{4}}$}$ 
 & $\text{\color{black}$\widetilde{m}_{\pi_{5}}$}$ 
 & $\text{\color{black}$\widetilde{m}_{\pi_{7}}$}$ 
 & $\text{\color{black}$\widetilde{m}_{\pi_{7}}$}$ 
 & $\text{\color{black}$\widetilde{m}_{\pi_{7}}$}$ 
 & $\text{\color{black}$\widetilde{m}_{\pi_{4}}$}$ 
 & $\text{\color{black}$\widetilde{m}_{\pi_{7}}$}$ 
 & $\text{\color{black}$\widetilde{m}_{\pi_{7}}$}$ 
 & $\text{\color{black}$\widetilde{m}_{\pi_{7}}$}$ 
 & $\text{\color{black}$\widetilde{m}_{\pi_{7}}$}$ 
 & $\text{\color{black}$\widetilde{m}_{\pi_{7}}$}$ 
 & $\text{\color{black}$\widetilde{m}_{\pi_{7}}$}$ \\
 $\text{\color{black}$\widetilde{m}_{\pi_{5}}$}$ 
 & $\text{\color{black}$\widetilde{m}_{\pi_{7}}$}$ 
 & $\text{\color{black}$\widetilde{m}_{\pi_{7}}$}$ 
 & $\text{\color{black}$\widetilde{m}_{\pi_{7}}$}$ 
 & $\text{\color{black}$\widetilde{m}_{\pi_{7}}$}$ 
 & $\text{\color{black}$\widetilde{m}_{\pi_{7}}$}$ 
 & $\text{\color{black}$\widetilde{m}_{\pi_{7}}$}$ 
 & $\text{\color{black}$\widetilde{m}_{\pi_{7}}$}$ 
 & $\text{\color{black}$\widetilde{m}_{\pi_{7}}$}$ 
 & $\text{\color{black}$\widetilde{m}_{\pi_{5}}$}$ 
 & $\text{\color{black}$\widetilde{m}_{\pi_{7}}$}$ 
 & $\text{\color{black}$\widetilde{m}_{\pi_{7}}$}$ 
 & $\text{\color{black}$\widetilde{m}_{\pi_{7}}$}$ 
 & $\text{\color{black}$\widetilde{m}_{\pi_{7}}$}$ 
 & $\text{\color{black}$\widetilde{m}_{\pi_{4}}$}$ 
 & $\text{\color{black}$\widetilde{m}_{\pi_{5}}$}$ \\
 $\text{\color{black}$\widetilde{m}_{\pi_{6}}$}$ 
 & $\text{\color{black}$\widetilde{m}_{\pi_{6}}$}$ 
 & $\text{\color{black}$\widetilde{m}_{\pi_{8}}$}$ 
 & $\text{\color{black}$\widetilde{m}_{\pi_{12}}$}$ 
 & $\text{\color{black}$\widetilde{m}_{\pi_{7}}$}$ 
 & $\text{\color{black}$\widetilde{m}_{\pi_{7}}$}$ 
 & $\text{\color{black}$\widetilde{m}_{\pi_{7}}$}$ 
 & $\text{\color{black}$\widetilde{m}_{\pi_{7}}$}$ 
 & $\text{\color{black}$\widetilde{m}_{\pi_{7}}$}$ 
 & $\text{\color{black}$\widetilde{m}_{\pi_{6}}$}$ 
 & $\text{\color{black}$\widetilde{m}_{\pi_{7}}$}$ 
 & $\text{\color{black}$\widetilde{m}_{\pi_{7}}$}$ 
 & $\text{\color{black}$\widetilde{m}_{\pi_{7}}$}$ 
 & $\text{\color{black}$\widetilde{m}_{\pi_{7}}$}$ 
 & $\text{\color{black}$\widetilde{m}_{\pi_{7}}$}$ 
 & $\text{\color{black}$\widetilde{m}_{\pi_{7}}$}$ \\ 
 $\text{\color{black}$\widetilde{m}_{\pi_{7}}$}$ 
 & $\text{\color{black}$\widetilde{m}_{\pi_{7}}$}$ 
 & $\text{\color{black}$\widetilde{m}_{\pi_{7}}$}$ 
 & $\text{\color{black}$\widetilde{m}_{\pi_{7}}$}$ 
 & $\text{\color{black}$\widetilde{m}_{\pi_{7}}$}$ 
 & $\text{\color{black}$\widetilde{m}_{\pi_{7}}$}$ 
 & $\text{\color{black}$\widetilde{m}_{\pi_{7}}$}$ 
 & $\text{\color{black}$\widetilde{m}_{\pi_{7}}$}$ 
 & $\text{\color{black}$\widetilde{m}_{\pi_{7}}$}$ 
 & $\text{\color{black}$\widetilde{m}_{\pi_{7}}$}$ 
 & $\text{\color{black}$\widetilde{m}_{\pi_{7}}$}$ 
 & $\text{\color{black}$\widetilde{m}_{\pi_{7}}$}$ 
 & $\text{\color{black}$\widetilde{m}_{\pi_{7}}$}$ 
 & $\text{\color{black}$\widetilde{m}_{\pi_{7}}$}$ 
 & $\text{\color{black}$\widetilde{m}_{\pi_{7}}$}$ 
 & $\text{\color{black}$\widetilde{m}_{\pi_{7}}$}$ \\
 $\text{\color{black}$\widetilde{m}_{\pi_{8}}$}$ 
 & $\text{\color{black}$\widetilde{m}_{\pi_{7}}$}$ 
 & $\text{\color{black}$\widetilde{m}_{\pi_{7}}$}$ 
 & $\text{\color{black}$\widetilde{m}_{\pi_{7}}$}$ 
 & $\text{\color{black}$\widetilde{m}_{\pi_{7}}$}$ 
 & $\text{\color{black}$\widetilde{m}_{\pi_{7}}$}$ 
 & $\text{\color{black}$\widetilde{m}_{\pi_{7}}$}$ 
 & $\text{\color{black}$\widetilde{m}_{\pi_{7}}$}$ 
 & $\text{\color{black}$\widetilde{m}_{\pi_{7}}$}$ 
 & $\text{\color{black}$\widetilde{m}_{\pi_{8}}$}$ 
 & $\text{\color{black}$\widetilde{m}_{\pi_{6}}$}$ 
 & $\text{\color{black}$\widetilde{m}_{\pi_{8}}$}$ 
 & $\text{\color{black}$\widetilde{m}_{\pi_{7}}$}$ 
 & $\text{\color{black}$\widetilde{m}_{\pi_{12}}$}$ 
 & $\text{\color{black}$\widetilde{m}_{\pi_{7}}$}$ 
 & $\text{\color{black}$\widetilde{m}_{\pi_{7}}$}$ \\
 $\text{\color{black}id}$ &
 $\text{\color{black}$\widetilde{m}_{\pi_{1}}$}$ &
 $\text{\color{black}$\widetilde{m}_{\pi_{2}}$}$ & 
 $\text{\color{black}$\widetilde{m}_{\pi_{3}}$}$ & 
 $\text{\color{black}$\widetilde{m}_{\pi_{4}}$}$ & 
 $\text{\color{black}$\widetilde{m}_{\pi_{5}}$}$ & 
 $\text{\color{black}$\widetilde{m}_{\pi_{6}}$}$ & 
 $\text{\color{black}$\widetilde{m}_{\pi_{7}}$}$ & 
 $\text{\color{black}$\widetilde{m}_{\pi_{8}}$}$ & 
 $\text{\color{black}id}$ &
 $\text{\color{black}$\widetilde{m}_{\pi_{10}}$}$ & 
 $\text{\color{black}$\widetilde{m}_{\pi_{11}}$}$ & 
 $\text{\color{black}$\widetilde{m}_{\pi_{12}}$}$ & 
 $\text{\color{black}$\widetilde{m}_{\pi_{13}}$}$ & 
 $\text{\color{black}$\widetilde{m}_{\pi_{14}}$}$ & 
 $\text{\color{black}$\widetilde{m}_{\pi_{15}}$}$ \\ 
 $\text{\color{black}$\widetilde{m}_{\pi_{10}}$}$ 
 & $\text{\color{black}$\widetilde{m}_{\pi_{10}}$}$ 
 & $\text{\color{black}$\widetilde{m}_{\pi_{11}}$}$ 
 & $\text{\color{black}$\widetilde{m}_{\pi_{13}}$}$ 
 & $\text{\color{black}$\widetilde{m}_{\pi_{7}}$}$ 
 & $\text{\color{black}$\widetilde{m}_{\pi_{7}}$}$ 
 & $\text{\color{black}$\widetilde{m}_{\pi_{7}}$}$ 
 & $\text{\color{black}$\widetilde{m}_{\pi_{7}}$}$ 
 & $\text{\color{black}$\widetilde{m}_{\pi_{7}}$}$ 
 & $\text{\color{black}$\widetilde{m}_{\pi_{10}}$}$ 
 & $\text{\color{black}$\widetilde{m}_{\pi_{7}}$}$ 
 & $\text{\color{black}$\widetilde{m}_{\pi_{7}}$}$ 
 & $\text{\color{black}$\widetilde{m}_{\pi_{7}}$}$ 
 & $\text{\color{black}$\widetilde{m}_{\pi_{7}}$}$ 
 & $\text{\color{black}$\widetilde{m}_{\pi_{7}}$}$ 
 & $\text{\color{black}$\widetilde{m}_{\pi_{7}}$}$ \\
 $\text{\color{black}$\widetilde{m}_{\pi_{11}}$}$ 
 & $\text{\color{black}$\widetilde{m}_{\pi_{7}}$}$ 
 & $\text{\color{black}$\widetilde{m}_{\pi_{7}}$}$ 
 & $\text{\color{black}$\widetilde{m}_{\pi_{7}}$}$ 
 & $\text{\color{black}$\widetilde{m}_{\pi_{7}}$}$ 
 & $\text{\color{black}$\widetilde{m}_{\pi_{7}}$}$ 
 & $\text{\color{black}$\widetilde{m}_{\pi_{7}}$}$ 
 & $\text{\color{black}$\widetilde{m}_{\pi_{7}}$}$ 
 & $\text{\color{black}$\widetilde{m}_{\pi_{7}}$}$ 
 & $\text{\color{black}$\widetilde{m}_{\pi_{11}}$}$ 
 & $\text{\color{black}$\widetilde{m}_{\pi_{10}}$}$ 
 & $\text{\color{black}$\widetilde{m}_{\pi_{11}}$}$ 
 & $\text{\color{black}$\widetilde{m}_{\pi_{7}}$}$ 
 & $\text{\color{black}$\widetilde{m}_{\pi_{13}}$}$ 
 & $\text{\color{black}$\widetilde{m}_{\pi_{7}}$}$ 
 & $\text{\color{black}$\widetilde{m}_{\pi_{7}}$}$ \\
 $\text{\color{black}$\widetilde{m}_{\pi_{12}}$}$ 
 & $\text{\color{black}$\widetilde{m}_{\pi_{7}}$}$ 
 & $\text{\color{black}$\widetilde{m}_{\pi_{7}}$}$ 
 & $\text{\color{black}$\widetilde{m}_{\pi_{7}}$}$ 
 & $\text{\color{black}$\widetilde{m}_{\pi_{7}}$}$ 
 & $\text{\color{black}$\widetilde{m}_{\pi_{7}}$}$ 
 & $\text{\color{black}$\widetilde{m}_{\pi_{6}}$}$ 
 & $\text{\color{black}$\widetilde{m}_{\pi_{7}}$}$ 
 & $\text{\color{black}$\widetilde{m}_{\pi_{8}}$}$ 
 & $\text{\color{black}$\widetilde{m}_{\pi_{12}}$}$ 
 & $\text{\color{black}$\widetilde{m}_{\pi_{7}}$}$ 
 & $\text{\color{black}$\widetilde{m}_{\pi_{7}}$}$ 
 & $\text{\color{black}$\widetilde{m}_{\pi_{12}}$}$ 
 & $\text{\color{black}$\widetilde{m}_{\pi_{7}}$}$ 
 & $\text{\color{black}$\widetilde{m}_{\pi_{7}}$}$ 
 & $\text{\color{black}$\widetilde{m}_{\pi_{7}}$}$ \\
 $\text{\color{black}$\widetilde{m}_{\pi_{13}}$}$ 
 & $\text{\color{black}$\widetilde{m}_{\pi_{7}}$}$ 
 & $\text{\color{black}$\widetilde{m}_{\pi_{7}}$}$ 
 & $\text{\color{black}$\widetilde{m}_{\pi_{7}}$}$ 
 & $\text{\color{black}$\widetilde{m}_{\pi_{7}}$}$ 
 & $\text{\color{black}$\widetilde{m}_{\pi_{7}}$}$ 
 & $\text{\color{black}$\widetilde{m}_{\pi_{10}}$}$ 
 & $\text{\color{black}$\widetilde{m}_{\pi_{7}}$}$ 
 & $\text{\color{black}$\widetilde{m}_{\pi_{11}}$}$ 
 & $\text{\color{black}$\widetilde{m}_{\pi_{13}}$}$ 
 & $\text{\color{black}$\widetilde{m}_{\pi_{7}}$}$ 
 & $\text{\color{black}$\widetilde{m}_{\pi_{7}}$}$ 
 & $\text{\color{black}$\widetilde{m}_{\pi_{13}}$}$ 
 & $\text{\color{black}$\widetilde{m}_{\pi_{7}}$}$ 
 & $\text{\color{black}$\widetilde{m}_{\pi_{7}}$}$ 
 & $\text{\color{black}$\widetilde{m}_{\pi_{7}}$}$ \\
 $\text{\color{black}$\widetilde{m}_{\pi_{14}}$}$ 
 & $\text{\color{black}$\widetilde{m}_{\pi_{7}}$}$ 
 & $\text{\color{black}$\widetilde{m}_{\pi_{7}}$}$ 
 & $\text{\color{black}$\widetilde{m}_{\pi_{7}}$}$ 
 & $\text{\color{black}$\widetilde{m}_{\pi_{14}}$}$ 
 & $\text{\color{black}$\widetilde{m}_{\pi_{15}}$}$ 
 & $\text{\color{black}$\widetilde{m}_{\pi_{7}}$}$ 
 & $\text{\color{black}$\widetilde{m}_{\pi_{7}}$}$ 
 & $\text{\color{black}$\widetilde{m}_{\pi_{7}}$}$ 
 & $\text{\color{black}$\widetilde{m}_{\pi_{14}}$}$ 
 & $\text{\color{black}$\widetilde{m}_{\pi_{7}}$}$ 
 & $\text{\color{black}$\widetilde{m}_{\pi_{7}}$}$ 
 & $\text{\color{black}$\widetilde{m}_{\pi_{7}}$}$ 
 & $\text{\color{black}$\widetilde{m}_{\pi_{7}}$}$ 
 & $\text{\color{black}$\widetilde{m}_{\pi_{7}}$}$ 
 & $\text{\color{black}$\widetilde{m}_{\pi_{7}}$}$ \\
 $\text{\color{black}$\widetilde{m}_{\pi_{15}}$}$ 
 & $\text{\color{black}$\widetilde{m}_{\pi_{7}}$}$ 
 & $\text{\color{black}$\widetilde{m}_{\pi_{7}}$}$ 
 & $\text{\color{black}$\widetilde{m}_{\pi_{7}}$}$ 
 & $\text{\color{black}$\widetilde{m}_{\pi_{7}}$}$ 
 & $\text{\color{black}$\widetilde{m}_{\pi_{7}}$}$ 
 & $\text{\color{black}$\widetilde{m}_{\pi_{7}}$}$ 
 & $\text{\color{black}$\widetilde{m}_{\pi_{7}}$}$ 
 & $\text{\color{black}$\widetilde{m}_{\pi_{7}}$}$ 
 & $\text{\color{black}$\widetilde{m}_{\pi_{15}}$}$ 
 & $\text{\color{black}$\widetilde{m}_{\pi_{7}}$}$ 
 & $\text{\color{black}$\widetilde{m}_{\pi_{7}}$}$ 
 & $\text{\color{black}$\widetilde{m}_{\pi_{7}}$}$ 
 & $\text{\color{black}$\widetilde{m}_{\pi_{7}}$}$ 
 & $\text{\color{black}$\widetilde{m}_{\pi_{14}}$}$ 
& $\text{\color{black}$\widetilde{m}_{\pi_{15}}$}$ 
 \end{tabular}
}
\caption{\label{Figuremulttable}
 The product operation on the monoid $\widetilde{M}_{2}$, writing $\text{id}$ in place of $\widetilde{m}_{\pi_{9}}$} 
\end{center}
\end{figure}

 To the best of our knowledge, our basis $\{ \widetilde{m}_{\pi} \}_{\pi}$ is the first known (and explicitly evaluated)
 basis of $\mathbb{C}A_2(n)$ that is closed 
 under the product operation on $\mathbb{C}A_2(n)$, and is the first known basis of $\mathbb{C}A_2(n)$ that has the structure of a 
 semigroup or a monoid under the multiplicative 
 operation on $\mathbb{C}A_2(n)$. This motivates our full solution to the problem given above on the 
 construction of a basis $\{ m_{\pi} \}_{\pi}$ satisfying the desired properties for semisimple partition algebras in full generality. 

\section{Background and preliminaries}\label{sectionbackground}
 The background/preliminary material given below is necessary for the purposes of our construction of the basis $\{ m_{\pi} \}_{\pi}$ 
 of $\mathbb{C}A_k(n)$. 

\subsection{Partition algebras}
 Let $n$ be a positive integer. With regard to the usual actions of $\text{GL}_{n}(\mathbb{C})$ and $S_{k}$ on $V^{\otimes k}$, for an 
 $n$-dimensional vector space $V$ over $\mathbb{C}$ (with $\text{GL}_{n}(\mathbb{C})$ acting diagonally on $V^{\otimes k}$ and with 
 $S_{k}$ acting via the place-permutation action), classical versions of Schur--Weyl duality give us that the action of $S_{k}$ generates 
 the centralizer algebra $\text{End}_{\text{GL}_{n}(\mathbb{C})}(V^{\otimes k})$ and vice-versa, referring to Halverson and Ram's paper 
 on partition algebras \cite{HalversonRam2005} and related references for details. More explicitly, what is meant by $S_{k}$ 
 generating $\text{End}_{\text{GL}_{n}(\mathbb{C})}\big( V^{\otimes k} \big)$ refers to the existince of a surjective morphism from $ 
 \mathbb{C}S_k$ to $\text{End}_{\text{GL}_{n}(\mathbb{C})}(V^{\otimes k})$ given by the natural representation of $\mathbb{C}S_k$, with 
 the $n \geq k$ case yielding the isomorphism 
\begin{equation}\label{CSkcong}
 \mathbb{C}S_k \cong \text{End}_{\text{GL}_{n}(\mathbb{C})}\big( V^{\otimes k} \big). 
\end{equation}

 Schur--Weyl duality is often considered as a centerpiece of or a cornerstone of the entire discipline of representation theory, 
 and this is due to how it relates the 
 irreducible representations of symmetric and general linear groups. This leads to research efforts based on 
 variants and generalizations of relations as in \eqref{CSkcong}, and partition algebras can be thought of as emerging in this way. 
 Explicitly, if $2k \leq n$, then the partition algebra $\mathbb{C}A_k(n)$ defined below satisfies 
\begin{equation}\label{CAisomorphic}
 \mathbb{C}A_k(n) \cong \text{End}_{S_{n}}\big( V^{\otimes k} \big), 
\end{equation}
 where the action of $S_{n}$ on $V^{\otimes k}$ is given by taking the action 
 of $\text{GL}_{n}(\mathbb{C})$ and by restricting it to the group of $n \times n$ permutation matrices. 

 For a positive integer $k$, let $A_{k}$ denote the set of all set-partitions of $\{ 1$, $2$, $ \ldots$, $k$, $1'$, $2'$, $ \ldots$, $ k' \}$. 
 For an element $\pi$ in $A_{k}$, we may denote $\pi$ with any graph $G$ on $\{ 1$, $2$, $ \ldots$, $k$, $ 1'$, $2'$, $\ldots$, $k' \}$ 
 such that the connected components of $G$ are the members of $\pi$. Any two graphs on $\{ 1$, $2$, $\ldots$, $ k$, $1'$, $2'$, $ 
 \ldots$, $ k' \}$ with the same connected components are considered to be equivalent, and we may define a \emph{partition 
 diagram} as an equivalence class associated with the given equivalence relation, and we may denote any such equivalence class 
 with any of its members, with the notational convention whereby the ``unprimed'' (resp.\ primed) vertices are arranged into a top 
 (resp. bottom) row according to the ordering $1 < 2 < \cdots < k < 1' < 2' < \cdots < k'$. We may also identify a given partition 
 diagram with its underlying set partition. 

\begin{example}\label{labelforprop}
 We may denote the set-partition
 $\{ \{ 5'$, $4 \}$, $ \{ 4'$, $1$, $2$, $3 \}$, 
 $ \{ 3'$, $ 1' \}$, $ \{ 2' \}$, $ \{ 5 \} \}$
 in $A_{5}$ as 
$$ 
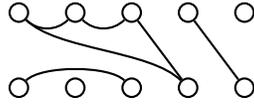
\begin{tikzpicture}[scale = 0.5,thick, baseline={(0,-1ex/2)}] 
\tikzstyle{vertex} = [shape = circle, minimum size = 7pt, inner sep = 1pt] 
\node[vertex] (G--5) at (6.0, -1) [shape = circle, draw] {}; 
\node[vertex] (G-4) at (4.5, 1) [shape = circle, draw] {}; 
\node[vertex] (G--4) at (4.5, -1) [shape = circle, draw] {}; 
\node[vertex] (G-1) at (0.0, 1) [shape = circle, draw] {}; 
\node[vertex] (G-2) at (1.5, 1) [shape = circle, draw] {}; 
\node[vertex] (G-3) at (3.0, 1) [shape = circle, draw] {}; 
\node[vertex] (G--3) at (3.0, -1) [shape = circle, draw] {}; 
\node[vertex] (G--1) at (0.0, -1) [shape = circle, draw] {}; 
\node[vertex] (G--2) at (1.5, -1) [shape = circle, draw] {}; 
\node[vertex] (G-5) at (6.0, 1) [shape = circle, draw] {}; 
\draw[] (G-4) .. controls +(0.75, -1) and +(-0.75, 1) .. (G--5); 
\draw[] (G-1) .. controls +(0.5, -0.5) and +(-0.5, -0.5) .. (G-2); 
\draw[] (G-2) .. controls +(0.5, -0.5) and +(-0.5, -0.5) .. (G-3); 
\draw[] (G-3) .. controls +(0.75, -1) and +(-0.75, 1) .. (G--4); 
\draw[] (G--4) .. controls +(-1, 1) and +(1, -1) .. (G-1); 
\draw[] (G--3) .. controls +(-0.6, 0.6) and +(0.6, 0.6) .. (G--1); 
\end{tikzpicture}$$
 and as 
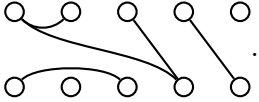
$$ \begin{tikzpicture}[scale = 0.5,thick, baseline={(0,-1ex/2)}] 
\tikzstyle{vertex} = [shape = circle, minimum size = 7pt, inner sep = 1pt] 
\node[vertex] (G--5) at (6.0, -1) [shape = circle, draw] {}; 
\node[vertex] (G-4) at (4.5, 1) [shape = circle, draw] {}; 
\node[vertex] (G--4) at (4.5, -1) [shape = circle, draw] {}; 
\node[vertex] (G-1) at (0.0, 1) [shape = circle, draw] {}; 
\node[vertex] (G-2) at (1.5, 1) [shape = circle, draw] {}; 
\node[vertex] (G-3) at (3.0, 1) [shape = circle, draw] {}; 
\node[vertex] (G--3) at (3.0, -1) [shape = circle, draw] {}; 
\node[vertex] (G--1) at (0.0, -1) [shape = circle, draw] {}; 
\node[vertex] (G--2) at (1.5, -1) [shape = circle, draw] {}; 
\node[vertex] (G-5) at (6.0, 1) [shape = circle, draw] {}; 
\draw[] (G-4) .. controls +(0.75, -1) and +(-0.75, 1) .. (G--5); 
\draw[] (G-1) .. controls +(0.5, -0.5) and +(-0.5, -0.5) .. (G-2); 
\draw[] (G-3) .. controls +(0.75, -1) and +(-0.75, 1) .. (G--4); 
\draw[] (G--4) .. controls +(-1, 1) and +(1, -1) .. (G-1); 
\draw[] (G--3) .. controls +(-0.6, 0.6) and +(0.6, 0.6) .. (G--1); 
\end{tikzpicture}.$$
\end{example} 

  A partition diagram is said to be of \emph{order $k$} if it is in $A_{k}$. An element in the set-partition associated with a
    partition diagram $\pi$ may be referred to as a \emph{block}, and a          connected component of $\pi$ may also be 
 referred to as a \emph{block}. A block of an order-$k$ partition diagram is said to be \emph{propagating} if it contains at least one 
 element in $\{ 1, 2, \ldots, k \}$ and at least one element in $\{ 1', 2', \ldots, k' \}$. The \emph{propagation number} of $\pi$ is 
 equal to the number of propagating blocks in $\pi$. 

\begin{example}
 The propagation number of the partition diagram in Example \ref{labelforprop} is $2$. 
\end{example}

 For partition diagrams $d_{1}$ and $d_{2}$ in $A_{k}$, let $d_{1} \circ d_{2}$ denote the partition diagram in $A_{k}$ obtained by 
 positioning $d_{1}$ on top of $d_{2}$ and by identifying the vertices in the lower row of $d_{1}$ with the vertices in the upper row of 
 $d_{2}$, and by then removing the vertices of the central row so as to preserve the relation given by uppermost vertices being in the 
 same component as lowermost vertices. The binary operation $\circ$ gives $A_{k}$ the structure of a monoid referred to as the 
 \emph{partition monoid}, and the identity element in $A_{k}$ is $\{ \{ 1, 1' \}, \{ 2, 2' \}, \ldots, \{ k, k' \} \}$. 

\begin{example}
 In the partition monoid $A_{5}$, the $\circ$-product of 
 $$ 
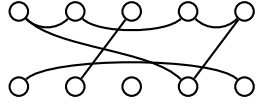
\begin{tikzpicture}[scale = 0.5,thick, baseline={(0,-1ex/2)}] 
\tikzstyle{vertex} = [shape = circle, minimum size = 7pt, inner sep = 1pt] 
\node[vertex] (G--5) at (6.0, -1) [shape = circle, draw] {}; 
\node[vertex] (G--1) at (0.0, -1) [shape = circle, draw] {}; 
\node[vertex] (G--4) at (4.5, -1) [shape = circle, draw] {}; 
\node[vertex] (G-1) at (0.0, 1) [shape = circle, draw] {}; 
\node[vertex] (G-2) at (1.5, 1) [shape = circle, draw] {}; 
\node[vertex] (G-4) at (4.5, 1) [shape = circle, draw] {}; 
\node[vertex] (G-5) at (6.0, 1) [shape = circle, draw] {}; 
\node[vertex] (G--3) at (3.0, -1) [shape = circle, draw] {}; 
\node[vertex] (G--2) at (1.5, -1) [shape = circle, draw] {}; 
\node[vertex] (G-3) at (3.0, 1) [shape = circle, draw] {}; 
\draw[] (G--5) .. controls +(-0.8, 0.8) and +(0.8, 0.8) .. (G--1); 
\draw[] (G-1) .. controls +(0.5, -0.5) and +(-0.5, -0.5) .. (G-2); 
\draw[] (G-2) .. controls +(0.6, -0.6) and +(-0.6, -0.6) .. (G-4); 
\draw[] (G-4) .. controls +(0.5, -0.5) and +(-0.5, -0.5) .. (G-5); 
\draw[] (G-5) .. controls +(-0.75, -1) and +(0.75, 1) .. (G--4); 
\draw[] (G--4) .. controls +(-1, 1) and +(1, -1) .. (G-1); 
\draw[] (G-3) .. controls +(-0.75, -1) and +(0.75, 1) .. (G--2); 
\end{tikzpicture} $$
 and 
 
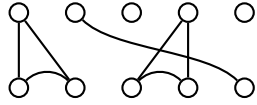
$$ \begin{tikzpicture}[scale = 0.5,thick, baseline={(0,-1ex/2)}] 
\tikzstyle{vertex} = [shape = circle, minimum size = 7pt, inner sep = 1pt] 
\node[vertex] (G--5) at (6.0, -1) [shape = circle, draw] {}; 
\node[vertex] (G-2) at (1.5, 1) [shape = circle, draw] {}; 
\node[vertex] (G--4) at (4.5, -1) [shape = circle, draw] {}; 
\node[vertex] (G--3) at (3.0, -1) [shape = circle, draw] {}; 
\node[vertex] (G-4) at (4.5, 1) [shape = circle, draw] {}; 
\node[vertex] (G--2) at (1.5, -1) [shape = circle, draw] {}; 
\node[vertex] (G--1) at (0.0, -1) [shape = circle, draw] {}; 
\node[vertex] (G-1) at (0.0, 1) [shape = circle, draw] {}; 
\node[vertex] (G-3) at (3.0, 1) [shape = circle, draw] {}; 
\node[vertex] (G-5) at (6.0, 1) [shape = circle, draw] {}; 
\draw[] (G-2) .. controls +(1, -1) and +(-1, 1) .. (G--5); 
\draw[] (G-4) .. controls +(0, -1) and +(0, 1) .. (G--4); 
\draw[] (G--4) .. controls +(-0.5, 0.5) and +(0.5, 0.5) .. (G--3); 
\draw[] (G--3) .. controls +(0.75, 1) and +(-0.75, -1) .. (G-4); 
\draw[] (G-1) .. controls +(0.75, -1) and +(-0.75, 1) .. (G--2); 
\draw[] (G--2) .. controls +(-0.5, 0.5) and +(0.5, 0.5) .. (G--1); 
\draw[] (G--1) .. controls +(0, 1) and +(0, -1) .. (G-1); 
\end{tikzpicture} $$
 is 
 
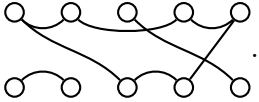
$$ \begin{tikzpicture}[scale = 0.5,thick, baseline={(0,-1ex/2)}] 
\tikzstyle{vertex} = [shape = circle, minimum size = 7pt, inner sep = 1pt] 
\node[vertex] (G--5) at (6.0, -1) [shape = circle, draw] {}; 
\node[vertex] (G-3) at (3.0, 1) [shape = circle, draw] {}; 
\node[vertex] (G--4) at (4.5, -1) [shape = circle, draw] {}; 
\node[vertex] (G--3) at (3.0, -1) [shape = circle, draw] {}; 
\node[vertex] (G-1) at (0.0, 1) [shape = circle, draw] {}; 
\node[vertex] (G-2) at (1.5, 1) [shape = circle, draw] {}; 
\node[vertex] (G-4) at (4.5, 1) [shape = circle, draw] {}; 
\node[vertex] (G-5) at (6.0, 1) [shape = circle, draw] {}; 
\node[vertex] (G--2) at (1.5, -1) [shape = circle, draw] {}; 
\node[vertex] (G--1) at (0.0, -1) [shape = circle, draw] {}; 
\draw[] (G-3) .. controls +(1, -1) and +(-1, 1) .. (G--5); 
\draw[] (G-1) .. controls +(0.5, -0.5) and +(-0.5, -0.5) .. (G-2); 
\draw[] (G-2) .. controls +(0.6, -0.6) and +(-0.6, -0.6) .. (G-4); 
\draw[] (G-4) .. controls +(0.5, -0.5) and +(-0.5, -0.5) .. (G-5); 
\draw[] (G-5) .. controls +(-0.75, -1) and +(0.75, 1) .. (G--4); 
\draw[] (G--4) .. controls +(-0.5, 0.5) and +(0.5, 0.5) .. (G--3); 
\draw[] (G--3) .. controls +(-1, 1) and +(1, -1) .. (G-1); 
\draw[] (G--2) .. controls +(-0.5, 0.5) and +(0.5, 0.5) .. (G--1); 
\end{tikzpicture}. $$
\end{example}

 For a nonnegative integer $k$, we let $A_{k + \frac{1}{2}}$ denote the submonoid of $A_{k+1}$ consisting of set-partitions $\pi$ in 
 $A_{k+1}$ such that $k + 1$ and $(k+1)'$ are in the same set in $\pi$. We henceforth let $n \in \mathbb{C}$, unless otherwise specified. 
 For $k \in \frac{1}{2} \mathbb{Z}_{>0}$, we may then define the \emph{partition algebra} $\mathbb{C}A_k(n)$ as the 
 $\mathbb{C}$-space spanned by $A_{k}$ endowed with the product operation whereby 
\begin{equation}\label{pimultdefinition}
 \pi_{1} \pi_{2} = n^{\ell(\pi_{1}, \pi_{2})} \, \pi_{1} \circ \pi_{2} 
\end{equation}
 for partition diagrams $\pi_1$ and $\pi_2$, where $\ell(\pi_{1}, \pi_{2})$ denotes the number of components contained entirely in the 
 middle row of the concatenation obtained by placing $\pi_{1}$ above $\pi_{2}$ and identifying the lower vertices of $\pi_{1}$ with the 
 upper vertices of $\pi_{2}$, and where 
 the multiplicative operation in \eqref{pimultdefinition} is extended 
 linearly. For the special case whereby $n$ is a natural number that satisfies $2k \leq n$, we obtain the isomorphic 
 equivalence in \eqref{CAisomorphic}. 

 The product rule in \eqref{pimultdefinition} gives rise to one of the two canonical bases of $\mathbb{C}A_k(n)$, and we adopt 
 the notational convention whereby, for a partition diagram or set-partition $\pi \in A_{k}$, when this is viewed as an element of 
 $ \mathbb{C}A_k(n)$, we may rewrite $\pi$ as $d_{\pi}$. This leads us to define \emph{diagram basis} of $\mathbb{C}A_k(n)$ 
 as $ \{ d_{\pi} \}_{\pi \in A_{k}} = \{ d_{\pi} \}_{\pi}$. 

\begin{example}\label{examplenotclosed}
 In the partition algebra $\mathbb{C}A_5(n)$, the product of 
 $$
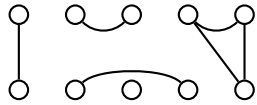
\begin{tikzpicture}[scale = 0.5,thick, baseline={(0,-1ex/2)}] 
\tikzstyle{vertex} = [shape = circle, minimum size = 7pt, inner sep = 1pt] 
\node[vertex] (G--5) at (6.0, -1) [shape = circle, draw] {}; 
\node[vertex] (G-4) at (4.5, 1) [shape = circle, draw] {}; 
\node[vertex] (G-5) at (6.0, 1) [shape = circle, draw] {}; 
\node[vertex] (G--4) at (4.5, -1) [shape = circle, draw] {}; 
\node[vertex] (G--2) at (1.5, -1) [shape = circle, draw] {}; 
\node[vertex] (G--3) at (3.0, -1) [shape = circle, draw] {}; 
\node[vertex] (G--1) at (0.0, -1) [shape = circle, draw] {}; 
\node[vertex] (G-1) at (0.0, 1) [shape = circle, draw] {}; 
\node[vertex] (G-2) at (1.5, 1) [shape = circle, draw] {}; 
\node[vertex] (G-3) at (3.0, 1) [shape = circle, draw] {}; 
\draw[] (G-4) .. controls +(0.5, -0.5) and +(-0.5, -0.5) .. (G-5); 
\draw[] (G-5) .. controls +(0, -1) and +(0, 1) .. (G--5); 
\draw[] (G--5) .. controls +(-0.75, 1) and +(0.75, -1) .. (G-4); 
\draw[] (G--4) .. controls +(-0.6, 0.6) and +(0.6, 0.6) .. (G--2); 
\draw[] (G-1) .. controls +(0, -1) and +(0, 1) .. (G--1); 
\draw[] (G-2) .. controls +(0.5, -0.5) and +(-0.5, -0.5) .. (G-3); 
\end{tikzpicture} $$
 and 
 
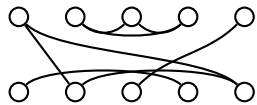
$$ \begin{tikzpicture}[scale = 0.5,thick, baseline={(0,-1ex/2)}] 
\tikzstyle{vertex} = [shape = circle, minimum size = 7pt, inner sep = 1pt] 
\node[vertex] (G--5) at (6.0, -1) [shape = circle, draw] {}; 
\node[vertex] (G--2) at (1.5, -1) [shape = circle, draw] {}; 
\node[vertex] (G-1) at (0.0, 1) [shape = circle, draw] {}; 
\node[vertex] (G--4) at (4.5, -1) [shape = circle, draw] {}; 
\node[vertex] (G--1) at (0.0, -1) [shape = circle, draw] {}; 
\node[vertex] (G--3) at (3.0, -1) [shape = circle, draw] {}; 
\node[vertex] (G-5) at (6.0, 1) [shape = circle, draw] {}; 
\node[vertex] (G-2) at (1.5, 1) [shape = circle, draw] {}; 
\node[vertex] (G-3) at (3.0, 1) [shape = circle, draw] {}; 
\node[vertex] (G-4) at (4.5, 1) [shape = circle, draw] {}; 
\draw[] (G-1) .. controls +(1, -1) and +(-1, 1) .. (G--5); 
\draw[] (G--5) .. controls +(-0.7, 0.7) and +(0.7, 0.7) .. (G--2); 
\draw[] (G--2) .. controls +(-0.75, 1) and +(0.75, -1) .. (G-1); 
\draw[] (G--4) .. controls +(-0.7, 0.7) and +(0.7, 0.7) .. (G--1); 
\draw[] (G-5) .. controls +(-1, -1) and +(1, 1) .. (G--3); 
\draw[] (G-2) .. controls +(0.5, -0.5) and +(-0.5, -0.5) .. (G-3); 
\draw[] (G-3) .. controls +(0.5, -0.5) and +(-0.5, -0.5) .. (G-4); 
\draw[] (G-4) .. controls +(-0.6, -0.6) and +(0.6, -0.6) .. (G-2); 
\end{tikzpicture} $$
 is the element 
 $$ n \, 
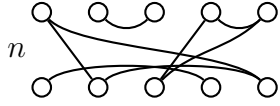
\begin{tikzpicture}[scale = 0.5,thick, baseline={(0,-1ex/2)}] 
\tikzstyle{vertex} = [shape = circle, minimum size = 7pt, inner sep = 1pt] 
\node[vertex] (G--5) at (6.0, -1) [shape = circle, draw] {}; 
\node[vertex] (G--2) at (1.5, -1) [shape = circle, draw] {}; 
\node[vertex] (G-1) at (0.0, 1) [shape = circle, draw] {}; 
\node[vertex] (G--4) at (4.5, -1) [shape = circle, draw] {}; 
\node[vertex] (G--1) at (0.0, -1) [shape = circle, draw] {}; 
\node[vertex] (G--3) at (3.0, -1) [shape = circle, draw] {}; 
\node[vertex] (G-4) at (4.5, 1) [shape = circle, draw] {}; 
\node[vertex] (G-5) at (6.0, 1) [shape = circle, draw] {}; 
\node[vertex] (G-2) at (1.5, 1) [shape = circle, draw] {}; 
\node[vertex] (G-3) at (3.0, 1) [shape = circle, draw] {}; 
\draw[] (G-1) .. controls +(1, -1) and +(-1, 1) .. (G--5); 
\draw[] (G--5) .. controls +(-0.7, 0.7) and +(0.7, 0.7) .. (G--2); 
\draw[] (G--2) .. controls +(-0.75, 1) and +(0.75, -1) .. (G-1); 
\draw[] (G--4) .. controls +(-0.7, 0.7) and +(0.7, 0.7) .. (G--1); 
\draw[] (G-4) .. controls +(0.5, -0.5) and +(-0.5, -0.5) .. (G-5); 
\draw[] (G-5) .. controls +(-1, -1) and +(1, 1) .. (G--3); 
\draw[] (G--3) .. controls +(0.75, 1) and +(-0.75, -1) .. (G-4); 
\draw[] (G-2) .. controls +(0.5, -0.5) and +(-0.5, -0.5) .. (G-3); 
\end{tikzpicture} $$
 in $\mathbb{C}A_5(n)$. 
\end{example}

\subsection{Vacillating tableaux}
 By analogy with how the branching rules associated with the irreducible representations for symmetric groups give rise to Young's 
 lattice, with the paths in Young's lattice giving rise to standard Young tableaux, the branching rules associated with the irreducible 
 representations for partition algebras give rise to a Bratteli diagram that Halverson and Ram denote as $\hat{A}$ 
 \cite{HalversonRam2005} and that is illustrated in Figure \ref{FigureBratteli}, and this gives rise to combinatorial objects known as 
 \emph{vacillating tableaux}. Our construction of the first known basis of $\mathbb{C}A_k(n)$ that gives $\mathbb{C}A_k(n)$ the 
 structure of a monoid algebra relies on the use of vacillating tableaux and is motivated by past research on combinatorial properties 
 of vacillating tableaux 
 \cite{BerikkyzyHarrisPunYanZhao2024Integers,BerikkyzyHarrisPunYanZhao2024Comb,HalversonLewandowski200406,Krattenthaleronline}. 

 An \emph{integer partition} $\lambda$ is a weakly decreasing tuple of positive integers. We let $\ell(\lambda)$ denote the number of 
 entries of $\lambda$, and we write $\lambda = (\lambda_{1}, \lambda_{2}, \ldots, \lambda_{\ell(\lambda)})$. The \emph{order} of 
 $ \lambda$ is equal to the sum of the entries of $\lambda$ and may be denoted with $|\lambda|$, and we adopt the convention 
 whereby the empty partition, denoted as $\varnothing$, is the unique integer partition of $0$. For a positive integer $n$, the relation 
 $|\lambda| = n$ may be denoted by writing $\lambda \vdash n$. We may identify a given integer partition $\lambda$ with a 
 tableau given by an arrangement of cells or boxes (placed within a grid) such that the number of cells in the $i^{\text{th}}$ row of this 
 tableau is $\lambda_{i}$ for $i \in \{ 1, 2, \ldots, \ell(\lambda) \}$. We adopt the so-called French convention for denoting tableaux of 
 partition shape, so that the initial entry is denoted as the lowest entry. This leads us toward Halverson and Ram's definition for 
 $\hat{A} $ \cite{HalversonRam2005}, as below. The vertices of $\hat{A}$ are arranged so as to form horizontal rows indexed by 
 nonnegative half-integers, and are labeled in the manner indicated below. 

\begin{enumerate}

\item Let $\hat{A}_{k}$ (resp.\ $\hat{A}_{k + \frac{1}{2}}$) denote the set of vertices at level $k$ (resp.\ $k + \frac{1}{2}$). The vertices in 
 $\hat{A}_{k}$ (resp.\ $\hat{A}_{k}$) are labeled with integer partitions $\mu$ such that $k - |\mu| \in \mathbb{Z}_{\geq 0}$, so that 
 $|\hat{A}_{k}|$ (resp.\ $\hat{A}_{k+\frac{1}{2}}$) is equal to the cardinality of the set of such integer partitions; 

\item There is an edge incident with $\lambda \in \hat{A}_{k}$ and $\mu \in \hat{A}_{k + \frac{1}{2}}$ if $\lambda = \mu$ or if $\mu$ 
 may be obtained from $\lambda$ by removing a box; and 
 
\item There is an edge incident with $\mu \in \hat{A}_{k + \frac{1}{2}}$ and $\lambda \in \hat{A}_{k + 1}$ if $\lambda = \mu$ or if 
 $ \lambda$ may be obtained from $\mu$ by adding a box. 
 
\end{enumerate}

\begin{figure}
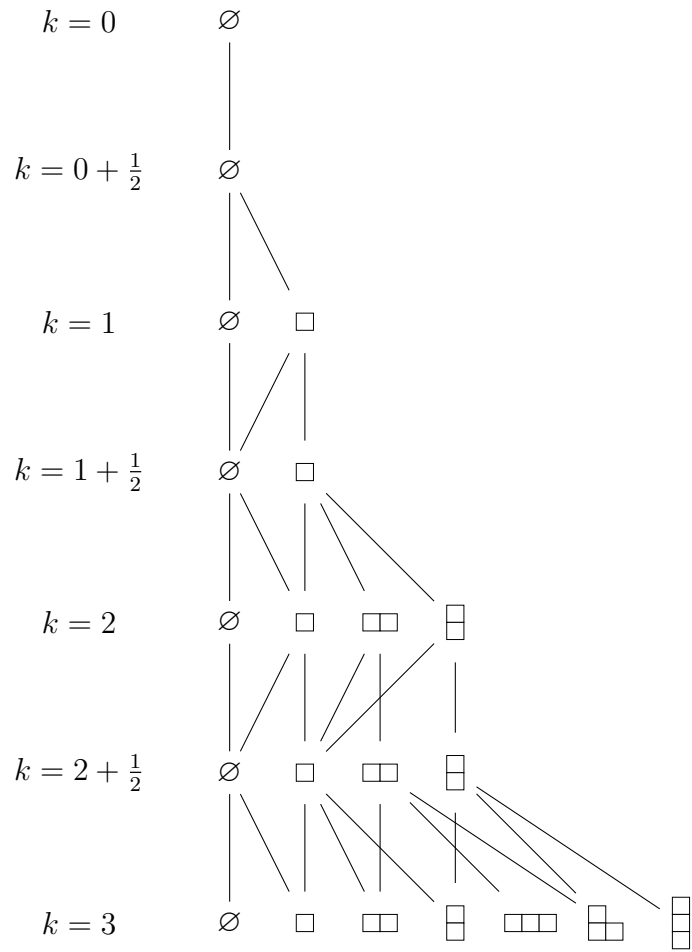

\begin{center}
\tikz {
 \begin{scope}[name prefix = levelfor0-]
 \node (A) at (0,12) {$k = 0$};
 \node (B) at (1,12) {\null};
 \node (C) at (2,12) {$\varnothing$};
 \end{scope}
 \begin{scope}[name prefix = unusedprefix1-]
 \node (unusedvertex1) at (0,11) {\null};
 \end{scope}
 \begin{scope}[name prefix = levelfor0half-]
 \node (A) at (0,10) {$k = 0 + \frac{1}{2}$};
 \node (B) at (1,10) {\null};
 \node (C) at (2,10) {$\varnothing$};
 \end{scope}
 \begin{scope}[name prefix = unusedprefix2-]
 \node (unusedvertex2) at (0,9) {\null};
 \end{scope}
 \begin{scope}[name prefix = levelfor1-]
 \node (A) at (0,8) {$k = 1$};
 \node (B) at (1,8) {\null};
 \node (C) at (2,8) {$\varnothing$};
 \node (D) at (3,8) {$ \Yvcentermath1 \Yboxdim{6.5pt} {\young(\null)}$};
 \end{scope}
 \begin{scope}[name prefix = unusedprefix3-]
 \node (unusedvertex3) at (0,7) {\null};
 \end{scope}
 \begin{scope}[name prefix = levelfor1half-]
 \node (A) at (0,6) {$k = 1 + \frac{1}{2}$};
 \node (B) at (1,6) {\null};
 \node (C) at (2,6) {$\varnothing$};
 \node (D) at (3,6) {$ \Yvcentermath1 \Yboxdim{6.5pt} {\young(\null)}$};
 \end{scope}
 \begin{scope}[name prefix = unusedprefix4-]
 \node (unusedvertex4) at (0,5) {\null};
 \end{scope}
 \begin{scope}[name prefix = levelfor2-]
 \node (A) at (0,4) {$k=2$};
 \node (B) at (1,4) {\null};
 \node (C) at (2,4) {$\varnothing$};
 \node (D) at (3,4) {$ \Yvcentermath1 \Yboxdim{6.5pt} {\young(\null)}$};
 \node (E) at (4,4) {$ \Yvcentermath1 \Yboxdim{6.5pt} {\young(\null\null)}$};
 \node (F) at (5,4) {$ \Yvcentermath1 \Yboxdim{6.5pt} {\young(\null,\null)}$};
 \end{scope}
 \begin{scope}[name prefix = unusedprefix5-]
 \node (unusedvertex5) at (0,3) {\null};
 \end{scope}
 \begin{scope}[name prefix = levelfor2half-]
 \node (A) at (0,2) {$k = 2 + \frac{1}{2}$};
 \node (B) at (1,2) {\null};
 \node (C) at (2,2) {$\varnothing$};
 \node (D) at (3,2) {$ \Yvcentermath1 \Yboxdim{6.5pt} {\young(\null)}$};
 \node (E) at (4,2) {$ \Yvcentermath1 \Yboxdim{6.5pt} {\young(\null\null)}$};
 \node (F) at (5,2) {$ \Yvcentermath1 \Yboxdim{6.5pt} {\young(\null,\null)}$};
 \end{scope}
 \begin{scope}[name prefix = unusedprefix6-]
 \node (unusedvertex6) at (0,1) {\null};
 \end{scope}
 \begin{scope}[name prefix = levelfor3-]
 \node (A) at (0,0) {$k = 3$};
 \node (B) at (1,0) {\null};
 \node (C) at (2,0) {$\varnothing$};
 \node (D) at (3,0) {$ \Yvcentermath1 \Yboxdim{6.5pt} {\young(\null)}$};
 \node (E) at (4,0) {$ \Yvcentermath1 \Yboxdim{6.5pt} {\young(\null\null)}$};
 \node (F) at (5,0) {$ \Yvcentermath1 \Yboxdim{6.5pt} {\young(\null,\null)}$};
 \node (G) at (6,0) {$ \Yvcentermath1 \Yboxdim{6.5pt} {\young(\null\null\null)}$};
 \node (H) at (7,0) {$ \Yvcentermath1 \Yboxdim{6.5pt} {\young(\null,\null\null)}$};
 \node (I) at (8,0) {$ \Yvcentermath1 \Yboxdim{6.5pt} {\young(\null,\null,\null)}$};
 \end{scope}

 \draw [black] (levelfor2half-C) -- (levelfor3-C);
 \draw [black] (levelfor1-C) -- (levelfor1half-C);
 \draw [black] (levelfor1half-C) -- (levelfor2-C);
 \draw [black] (levelfor2-C) -- (levelfor2half-C);
 \draw [black] (levelfor0half-C) -- (levelfor1-C);
 \draw [black] (levelfor0-C) -- (levelfor0half-C);
 \draw [black] (levelfor2half-C) -- (levelfor3-D);
 \draw [black] (levelfor2half-D) -- (levelfor3-E);
 \draw [black] (levelfor2half-D) -- (levelfor3-D);
 \draw [black] (levelfor2half-E) -- (levelfor3-E);
 \draw [black] (levelfor2half-F) -- (levelfor3-F);
 \draw [black] (levelfor2half-D) -- (levelfor2-D);
 \draw [black] (levelfor1half-D) -- (levelfor2-D);
 \draw [black] (levelfor1half-D) -- (levelfor1-D);
 \draw [black] (levelfor0half-C) -- (levelfor1-D);
 \draw [black] (levelfor1half-C) -- (levelfor1-D);
 \draw [black] (levelfor1half-C) -- (levelfor2-D);
 \draw [black] (levelfor2half-C) -- (levelfor2-D);
 \draw [black] (levelfor2half-E) -- (levelfor2-E);
 \draw [black] (levelfor2half-F) -- (levelfor2-F);
 \draw [black] (levelfor1half-D) -- (levelfor2-E);
 \draw [black] (levelfor1half-D) -- (levelfor2-F);
 \draw [black] (levelfor2half-D) -- (levelfor2-E);
 \draw [black] (levelfor2half-D) -- (levelfor2-F);
 \draw [black] (levelfor2half-D) -- (levelfor3-F);
 \draw [black] (levelfor2half-E) -- (levelfor3-G);
 \draw [black] (levelfor2half-E) -- (levelfor3-H);
 \draw [black] (levelfor2half-F) -- (levelfor3-H);
 \draw [black] (levelfor2half-F) -- (levelfor3-I);
}

\caption{\label{FigureBratteli}
 An illustration of the Bratteli diagram $\hat{A}$.} 
\end{center}
\end{figure}

 A \emph{vacillating tableau} refers to a finite tuple of integer partitions given by the consecutive vertices in a path in $\hat{A}$ starting 
 from level $0$. If this path ends on level $k$ of $\hat{A}$, then we refer to the corresponding vacillating tableau as being \emph{of 
 order $k$}. A fundamental result in the representation theory of partition algebras gives us that the irreducible representations of $ 
 \mathbb{C}A_k(n)$ (for the semisimple case) are indexed by $\hat{A}_{k}$ and that the dimension/multiplicity of a simple 
 $\mathbb{C} A_k(n)$-module corresponding to $\lambda \in \hat{A}_{k}$ is equal to the number of vacillating tableaux of of order $k$ 
 ending with $\lambda$. 

\begin{example}
 With reference to Figure \ref{FigureBratteli}, we find that there are $2$ vacillating tableaux corresponding to paths ending at 
 $\varnothing$ at level $k = 2$, namely 
\begin{equation*}
 \big( \varnothing, \varnothing, \varnothing, \varnothing, \varnothing \big) \ \ \ \text{and} \ \ \ \big( \varnothing, \varnothing, 
 \Yvcentermath1 \Yboxdim{6.5pt} {\young(\null)}, \varnothing, \varnothing \big). 
 \end{equation*}
 Similarly, we find that there are $3$ vacillating tableaux corresponding to paths ending at 
 $ \Yvcentermath1 \Yboxdim{6.5pt} {\young(\null)} $ at level $k = 2$, namely 
\begin{equation*}
 \big( \varnothing, \varnothing, \varnothing, \varnothing, 
 \Yvcentermath1 \Yboxdim{6.5pt} {\young(\null)} \big), \ \ \ 
 \big( \varnothing, \varnothing, \Yvcentermath1 \Yboxdim{6.5pt} {\young(\null)}, \varnothing, 
 \Yvcentermath1 \Yboxdim{6.5pt} {\young(\null)} \big), \ \ \ 
 \text{and} \ \ \ \big( \varnothing, \varnothing, 
 \Yvcentermath1 \Yboxdim{6.5pt} {\young(\null)}, \Yvcentermath1 \Yboxdim{6.5pt} {\young(\null)}, 
 \Yvcentermath1 \Yboxdim{6.5pt} {\young(\null)} \big). 
 \end{equation*}
 Also, there is a unique vacillating tableau corresponding to the path ending at 
 $ \Yvcentermath1 \Yboxdim{6.5pt} {\young(\null\null)} $ at level $k = 2$, namely 
\begin{equation*}
 \big( \varnothing, \varnothing, 
 \Yvcentermath1 \Yboxdim{6.5pt} {\young(\null)},
 \Yvcentermath1 \Yboxdim{6.5pt} {\young(\null)}, 
 \Yvcentermath1 \Yboxdim{6.5pt} {\young(\null\null)} \big). 
 \end{equation*}
 Finally, there is a unique vacillating tableau corresponding to the path ending at 
 $ \Yvcentermath1 \Yboxdim{6.5pt} {\young(\null,\null)} $ at level $k = 2$, namely 
\begin{equation*}
 \big( \varnothing, \varnothing, 
 \Yvcentermath1 \Yboxdim{6.5pt} {\young(\null)},
 \Yvcentermath1 \Yboxdim{6.5pt} {\young(\null)}, 
 \Yvcentermath1 \Yboxdim{6.5pt} {\young(\null,\null)} \big). 
 \end{equation*}
 The enumerations given above can be thought of as providing a decomposition of the dimension 
 of $\mathbb{C}A_2(n)$ based on the dimensions and multiplicities of the 
 simple $\mathbb{C}A_2(n)$-modules, writing 
 $$ \dim \, \mathbb{C}A_2(n) = 2^2 + 3^2 + 1^2 + 1^2. $$
\end{example}

\subsection{Halverson and Ram's matrix unit construction}
 For a finite-dimensional algebra, this algebra is said to be \emph{semisimple} if it can be expressed, up to isomorphism, as a direct sum 
 of full matrix algebras. Let $A$ be a finite-dimensional and semisimple algebra, and let $\varphi$ denote any isomorphism from a 
 direct sum of full matrix algebras to $A$. A \emph{matrix unit basis} for $A$ may then be defined as the image under $\varphi$ 
 of the basis $U$ of the domain of $\varphi$ such that every element in $U$ may be written as a 
 \emph{matrix unit}, i.e., 
 as a matrix with a unique $1$-entry and with $0$-entries everywhere else. 
 Elements in any matrix unit basis of $A$ may be referred to as \emph{matrix units} of $A$. 

 The importance of obtaining explicit evaluations for elements in matrix unit bases for semisimple algebras related to symmetric 
 group algebras is reflected by \emph{Young's construction} of matrix unit bases for symmetric group algebras, in view of the 
 importance of Young's construction within both representation theory and combinatorics. The first known construction of matrix unit 
 bases for partition algebras is due to Halverson and Ram \cite{HalversonRam2005} and relies on what is 
 referred to as the \emph{basic construction} given by Bourbaki 
 \cite[\S2]{Bourbaki1990} and is to be of key importance in our work. 

 Given any matrix unit basis $B$ of $\mathbb{C}S_k$, the Halverson--Ram matrix unit construction provides a matrix unit 
 basis $\{ e^{\lambda}_{v_{1}, v_{2}} \}_{\lambda, v_1, v_2}$ of $\mathbb{C}A_k(n)$ 
 defined in terms of $B$, where $v_1$ and $v_2$ 
 are vacillating tableaux of order $k$ 
 ending with the same partition $\lambda$ \cite{HalversonRam2005}. The elements of the basis $\{ e^{\lambda}_{v_{1}, v_{2}} 
 \}_{\lambda, v_1, v_2}$ satisfy the matrix unit multiplication rules such that 
\begin{equation*}
 e_{w_1, w_2}^{\mu} e_{w_{3}, w_{4}}^{\nu} = \begin{cases} 
 0 & \text{if $\mu \neq \nu$}, \\ 
 0 & \text{if $w_2 \neq w_3$, and} \\ 
 e_{w_1, w_4}^{\mu} & \text{if $w_2 = w_3$ and $\mu = \nu$.} 
 \end{cases} 
\end{equation*}
 Since a basis $B$ of the specified form is required in Halverson and Ram's construction, we review Young's construction of symmetric 
 group algebra matrix units, referring to Garsia and E\u gecio\u glu's monograph \cite[\S1]{GarsiaEgecioglu2020} for further 
 preliminaries related to Young's construction. 
 
 A \emph{standard Young tableau} is a tableau $T$ of partition shape $\lambda$ with increasing rows and columns such that 
 the set of the labels of $T$ is $\{ 1, 2, \ldots, |\lambda| \}$. We let standard Young tableaux of the same 
 shape be ordered according to Young's First Letter Order relation
 $<_{YFLO}$ whereby one tableau $T_{1}$ precedes another $T_{2}$ if the first entry of 
 disagreement between $T_{1}$ and $T_{2}$ is lower in $T_{1}$. 
 The expression $f^{\lambda}$ denotes the number of standard Young tableaux of shape $\lambda$. We let 
 the tableaux of this form be denoted by writing $$ S_{1}^{\lambda} <_{YFLO} S_{2}^{\lambda} <_{YFLO} \cdots <_{YFLO} 
 S_{f^{\lambda}}^{\lambda}. $$ 

 The \emph{row} (resp.\ \emph{column}) \emph{group} $R(T)$ (resp.\ $C(T)$) for a Young tableau $T$ is the group 
 consisting of permutations on the labels of $T$ that leave the rows (resp.\ columns) unaltered up to rearrangements of the labels of a 
 given row (resp.\ column). This leads us to write $$ P(T) = \sum_{\alpha \in R(T)} \alpha \ \ \ 
 \text{and} \ \ \ N(T) = \sum_{\beta \in C(T)} \text{sign}(\beta) \beta. $$ We may then define Young's $ \gamma$-elements so that $$ 
 \gamma_{i}^{\lambda} = \frac{f^{\lambda}}{n!} N\left( S_{i}^{\lambda} \right) P\left( S_{i}^{\lambda} \right). $$ For two tableaux $T_{1} 
 $ and $T_{2}$ of the same shape and with the same sets of labels and without any repeated labels, 
 we let $\sigma_{T_{1}, T_{2}}$ denote the unique permutation such 
 that $T_{1} = \sigma_{T_{1}, T_{2}} T_{2}$, where permutations act on tableaux by permuting the labels, and we write 
 $\sigma_{S_{i}^{\lambda}, S_{j}^{\lambda}} = \sigma_{i, j}^{\lambda}$. 

 We are now in a position to provide Young's matrix unit formula whereby 
\begin{equation}\label{Youngmainformula}
 s_{i, j}^{\lambda} = \sigma_{i, j}^{\lambda} \gamma_{j}^{\lambda} 
 \left( 1 - \gamma_{j+1}^{\lambda} \right) 
 \left( 1 - \gamma_{j+2}^{\lambda} \right) \cdots 
 \left( 1 - \gamma_{f^{\lambda}}^{\lambda} \right). 
\end{equation}
 A remarkable result due to Young gives us that for every positive integer $k$, the family 
\begin{equation}\label{Youngsbasis}
 \{ s_{i, j}^{\lambda} : \lambda \vdash k, 1 \leq i, j \leq f^{\lambda} \} 
\end{equation}
 is a matrix unit basis of $ \mathbb{C}S_{k}$.

\begin{example}
 Let permutations be denoted with two-line notation. We may verify that 
\begin{equation}\label{Younghook11}
 s^{\Yvcentermath1 \Yboxdim{6.5pt} {\young(\null,\null\null)}}_{1, 1} 
 = \frac{1}{3} \bigl(\begin{smallmatrix} 
 1 & 2 & 3 \\ 1 & 2 & 3 
 \end{smallmatrix} \bigr) 
 - \frac{1}{3} \bigl(\begin{smallmatrix} 
 1 & 2 & 3 \\ 2 & 1 & 3 
 \end{smallmatrix} \bigr) 
 - \frac{1}{3} 
 \bigl(\begin{smallmatrix}
 1 & 2 & 3 \\ 3 & 1 & 2 
 \end{smallmatrix} \bigr) 
 + \frac{1}{3} \bigl(\begin{smallmatrix}
1 & 2 & 3 \\ 3 & 2 & 1 
\end{smallmatrix} \bigr), 
 \end{equation}
 and we may verify that the right-hand side of \eqref{Younghook11} is idempotent. We may also verify the evaluations such that 
\begin{equation*}
 s^{\Yvcentermath1 \Yboxdim{6.5pt} {\young(\null,\null\null)}}_{1, 2} 
 = \frac{1}{3} \bigl(\begin{smallmatrix} 
 1 & 2 & 3 \\ 1 & 3 & 2 
 \end{smallmatrix} \bigr) 
 - \frac{1}{3} \bigl(\begin{smallmatrix} 
 1 & 2 & 3 \\ 2 & 3 & 1 
 \end{smallmatrix} \bigr) 
 + \frac{1}{3} 
 \bigl(\begin{smallmatrix}
 1 & 2 & 3 \\ 3 & 1 & 2 
 \end{smallmatrix} \bigr) 
 - \frac{1}{3} \bigl(\begin{smallmatrix}
1 & 2 & 3 \\ 3 & 2 & 1 
\end{smallmatrix} \bigr) 
 \end{equation*}
 and such that 
\begin{equation*}
 s^{\Yvcentermath1 \Yboxdim{6.5pt} {\young(\null,\null\null)}}_{2, 1} 
 = \frac{1}{3} \bigl(\begin{smallmatrix} 
 1 & 2 & 3 \\ 1 & 3 & 2 
 \end{smallmatrix} \bigr) 
 - \frac{1}{3} \bigl(\begin{smallmatrix} 
 1 & 2 & 3 \\ 2 & 1 & 3 
 \end{smallmatrix} \bigr) 
 + \frac{1}{3} 
 \bigl(\begin{smallmatrix}
 1 & 2 & 3 \\ 2 & 3 & 1 
 \end{smallmatrix} \bigr) 
 - \frac{1}{3} \bigl(\begin{smallmatrix}
1 & 2 & 3 \\ 3 & 1 & 2 
\end{smallmatrix} \bigr), 
 \end{equation*}
 and we may use the above evaluations to check that 
 $$ s^{\Yvcentermath1 \Yboxdim{6.5pt} {\young(\null,\null\null)}}_{1, 2} 
 s^{\Yvcentermath1 \Yboxdim{6.5pt} {\young(\null,\null\null)}}_{2, 1} 
 = s^{\Yvcentermath1 \Yboxdim{6.5pt} {\young(\null,\null\null)}}_{1, 1} 
 \ \ \ \text{and} \ \ \ 
 s^{\Yvcentermath1 \Yboxdim{6.5pt} {\young(\null,\null\null)}}_{1, 2} 
 s^{\Yvcentermath1 \Yboxdim{6.5pt} {\young(\null,\null\null)}}_{1, 2} 
 = 0 \ \ \ \text{and} \ \ \ 
 s^{\Yvcentermath1 \Yboxdim{6.5pt} {\young(\null,\null\null)}}_{2, 1} 
 s^{\Yvcentermath1 \Yboxdim{6.5pt} {\young(\null,\null\null)}}_{2, 1} 
 = 0, $$ as desired. 
\end{example}

 Letting it be understood that we are working with a given partition algebra $\mathbb{C}A_k(n)$, we write $p_{i + \frac{1}{2}}$ in 
 place of the partition diagram corresponding to $\{ \{ 1, 1' \}, \ldots, \{ i-1, (i-1)' \}, \{ i, i+1, i', (i+1)' \}, \{ i+2, (i+2)' \}, \ldots, \{ k, 
 k' \} \}$, and we write $p_{i}$ in place of the partition diagram corresponding to $ \{ \{ 1, 1' \}, \ldots, \{ i-1, (i-1)' \}, \{ i \}, \{ i' 
 \}, \{ i + 1, (i + 1)' \}, \ldots, \{ k, k' \} \}$. For two elements $a_{1}$ and $a_{2}$ in $\mathbb{C}A_k(n)$, let the equivalence relation $\sim$ 
 be such that $a_{1} \sim a_{2}$ if and only if $a_{2}$ is a nonzero scalar multiple of $a_{1}$. We also let $\hat{A}_{k}^{\mu}$ denote the 
 set of order-$k$ vacillating tableaux corresponding to paths in $\hat{A}$ ending on $\mu$. 

 Let $\{ e^{\lambda}_{v_{1}, v_{2}} \}_{\lambda, v_1, v_2}$ denote the matrix unit basis of $\mathbb{C}A_k(n)$ 
 obtained through the application of Halverson and Ram's construction \cite{HalversonRam2005} 
 with the use of the matrix unit basis for $\mathbb{C}S_{k}$ in \eqref{Youngsbasis} 
 (and, as later clarified, with the use of vacillating tableaux of the form indicated in Definition \ref{Tfun} below), 
 again letting $v_{1}$ and $v_{2}$ be order-$k$ vacillating tableaux ending with the same partition $\lambda$. For the sake of 
 convenience, since the superscript of $ e^{\lambda}_{v_{1}, v_{2}}$ is determined by $v_{1}$ and $v_{2}$, we may 
 simplify our notation by 
 writing $ e_{v_{1}, v_{2}}$ in place of $ e^{\lambda}_{v_{1}, v_{2}}$. For a tuple $\mathcal{T}$, let $\mathcal{T}^{-}$ denote the tuple 
 obtained by removing the final entry of $\mathcal{T}$. For a partition $\lambda \in \hat{A}_{\ell}$ such that $\lambda \vdash \ell$, and 
 for a vacillating tableau $P \in \hat{A}_{\ell}^{\lambda}$, we may identify $\lambda$ with the standard Young tableau obtained by 
 following the same path (with the same sequence of vertices) in Young's lattice, relative to the path in $\hat{A}$ corresponding to $P$. 

 Our construction of the basis $\{ m_{\pi} \}_{\pi}$ of $\mathbb{C}A_k(n)$ can be adapted so as to be applicable with any family of 
 matrix units for $\mathbb{C}A_k(n)$. For the purposes of producing explicit demonstrations of our construction, as in Section 
 \ref{subsectionmotivating}, it is often more convenient to make use of a variant $\{ c_{v_{1}, v_{2}} e_{v_{1}, 
 v_{2}} \}_{v_1, v_2}$ of the basis $\{ e_{v_1, v_2} \}_{v_1, v_2} = \{ e_{v_{1}, v_{2}}^{\lambda} \}_{\lambda, v_1, v_2}$
 indicated above, 
 for nonzero scalars $ c_{v_{1}, v_{2}}$ such that 
 $\{ c_{v_{1}, v_{2}} e_{v_{1}, v_{2}} \}_{v_1, v_2}$ is a matrix unit basis of $\mathbb{C}A_k(n)$ (noting 
 that it is necessarily the case that $c_{v_{1}, v_{1}} = 1$). Because of this, and for the sake of brevity, we simplify the below 
 review of Halverson and Ram's construction through the use of equivalence classes with respect to $\sim$. 
 However, our construction of the basis $\{ m_{\pi} \}_{\pi}$ is explicit and uses the explicit version of the 
 Halverson--Ram matrix units, referring to the work of Halverson and Ram \cite{HalversonRam2005} on the 
 evaluation of the required 
 scalars, which are given in terms of the irreducible characters for partition algebras. 

 Halverson and Ram's construction of matrix units for 
 $\mathbb{C}A_{\ell}(n)$ relies on a recursion that gives us, for 
 $ \mu \in \hat{A}_{\ell}$ and $|\mu| \leq \ell - 1$ and $P, Q \in \hat{A}_{\ell}^{\mu}$, that 
\begin{equation}\label{HRrecsim}
 e_{P, Q} \sim e_{P^{-}, T} \, p_{\ell} \, e_{T, Q^{-}}, 
\end{equation}
 where $T$ can be chosen as any element of $\hat{A}_{\ell-\frac{1}{2}}^{\mu}$. For the purposes of our definition given below
 for \emph{Young--Halverson--Ram matrix units}, 
 letting $\lambda \vdash \ell$ and letting $P$ and $Q$ 
 be in $\hat{A}_{\ell}^{\lambda}$, we then set 
\begin{equation}\label{eintermss}
 e_{P, Q} = \left( 1 - z \right) s_{P, Q}^{\lambda} 
\end{equation}
 for 
\begin{equation}\label{zdoublesum}
 z = \sum_{{\substack{ \mu \in 
 \hat{A}_{\ell} \\ |\mu| \leq \ell - 1 }}} \sum_{P \in \hat{A}_{\ell}^{\mu}} e_{P, P}. 
\end{equation}
 An explicit version, with equality as opposed to 
 equality up to $\sim$, of the relation in \eqref{HRrecsim} is given by Halverson and Ram 
 \cite[Theorem 2.26]{HalversonRam2005}, and we invite the interested reader to review the given Halverson--Ram reference for the 
 explicit version of \eqref{HRrecsim}. 
 The family $$ \{ e_{P, Q} : \text{$P, Q \in \hat{A}_{\ell}^{\mu}$ for $|\mu| \leq \ell$} \} $$ is then a 
 matrix unit basis of $\mathbb{C}A_{\ell}(n)$ \cite{HalversonRam2005}, 
 so that the matrix unit multiplication rules are satisfied, i.e., so that $e_{P_{1}, P_{2}} e_{P_{3}, P_{4}}$ vanishes if 
 $P_{2} \neq P_{3}$, with $e_{P_{1}, P_{2}} e_{P_{3}, P_{4}} = e_{P_{1}, P_{4}} $ otherwise. 

 To obtain a matrix unit basis for $\mathbb{C}A_{\ell}(n)$ in an unambiguous way, we fix a choice of $T$, in view of 
 Halverson and Ram's recursion corresponding to \eqref{HRrecsim}. 

\begin{definition}\label{Tfun}
 For any vacillating tableau $\mathcal{R}$ ending on level $\ell$ in $\hat{A}$ such that the final entry in $\mathcal{R}$ is of order less than 
 or equal to $\ell - 1$, we set $T = T(\mathcal{R})$ as the vacillating tableau of length one less than the length of $\mathcal{R}$ 
 and obtained in the following manner. 

\begin{enumerate}

\item The last entry of $T$ is the same as the second-to-last entry of $\mathcal{R}$; 

\item The vacillating tableau obtained by removing the last entry of $T$ is in $\hat{A}_{\ell - 1}^{\mu}$, where $\mu$ is the last entry 
 of $\mathcal{R}$; 

\item The following entries of $T$, in reverse order after its second-to-last entry $\mu$, are given by successive applications (as many 
 times as possible) of steps (a) and (b) below 
 to plot out a path in the Bratteli diagram $\hat{A}$, starting in such a way so that the 
 third-to-last entry of $T$ is the the partition obtained by applying step (a) below 
 to $\mu$: 

\begin{enumerate}

\item Remove a box in the northernmost row so as to move upwards to the next level in $\hat{A}$; and 

\item Then move upwards to the next level in $\hat{A}$ and to a copy of the partition obtained in step (a); 

\end{enumerate}

\item Once two copies of ${\Yvcentermath1 \Yboxdim{6.5pt} {\young(\null)}}$ have been reached or the northernmost 
 copy of ${\Yvcentermath1 \Yboxdim{6.5pt} {\young(\null)}}$ has been reached or a copy of $\varnothing$ has been reached, 
 follow the unique path to the top of the Bratteli diagram along its leftmost column. 

\end{enumerate}
\end{definition}

\begin{example}
 According to Definition \ref{Tfun}, we have that 
 $$ T\big( \varnothing, \varnothing, 
 {\Yvcentermath1 \Yboxdim{6.5pt} {\young(\null)}}, 
 {\Yvcentermath1 \Yboxdim{6.5pt} {\young(\null)}}, 
 {\Yvcentermath1 \Yboxdim{6.5pt} {\young(\null,\null)}}, 
 {\Yvcentermath1 \Yboxdim{6.5pt} {\young(\null)}}, 
 {\Yvcentermath1 \Yboxdim{6.5pt} {\young(\null,\null)}} \big) 
 = \big( \varnothing, \varnothing, 
 {\Yvcentermath1 \Yboxdim{6.5pt} {\young(\null)}}, 
 {\Yvcentermath1 \Yboxdim{6.5pt} {\young(\null)}}, 
 {\Yvcentermath1 \Yboxdim{6.5pt} {\young(\null,\null)}}, 
 {\Yvcentermath1 \Yboxdim{6.5pt} {\young(\null)}} 
 \big), $$ 
 whereas $$ T\big( \varnothing, \varnothing, \varnothing, \varnothing, 
 {\Yvcentermath1 \Yboxdim{6.5pt} {\young(\null)}} \big) 
 = \big( \varnothing, \varnothing, {\Yvcentermath1 \Yboxdim{6.5pt} {\young(\null)}}, \varnothing \big). $$
\end{example}

\begin{definition}
 The \emph{Young--Halverson--Ram matrix unit basis} is the basis of $\mathbb{C}A_{\ell}$ obtained according to the explicit version of 
 \eqref{HRrecsim}, for the case whereby the vacillating tableau $T$ is as in Definition \ref{Tfun}, 
 together with \eqref{eintermss} and \eqref{zdoublesum}. 
\end{definition}

\begin{example}\label{exampleYHR}
 Under the assumption that $n \in \mathbb{C} \setminus \{ 0, 1, \ldots, 2k-2 \}$ so that $\mathbb{C}A_k(n)$ is semisimple, for the $k 
 = 2$ case, we have that 
\begin{multline*}
 e_{\big( \varnothing, \varnothing, {\Yvcentermath1 \Yboxdim{6.5pt} {\young(\null)}},
 \varnothing, \varnothing \big), \big( \varnothing, \varnothing, {\Yvcentermath1 \Yboxdim{6.5pt} {\young(\null)}},
 \varnothing, \varnothing \big)} 
 = \\ \frac{1}{n^{2}(n - 1)} \begin{tikzpicture}[scale = 0.5,thick, baseline={(0,-1ex/2)}] 
\tikzstyle{vertex} = [shape = circle, minimum size = 7pt, inner sep = 1pt] 
\node[vertex] (G--2) at (1.5, -1) [shape = circle, draw] {}; 
\node[vertex] (G--1) at (0.0, -1) [shape = circle, draw] {}; 
\node[vertex] (G-1) at (0.0, 1) [shape = circle, draw] {}; 
\node[vertex] (G-2) at (1.5, 1) [shape = circle, draw] {}; 
\end{tikzpicture} 
 - \frac{1}{n(n - 1)} \begin{tikzpicture}[scale = 0.5,thick, baseline={(0,-1ex/2)}] 
\tikzstyle{vertex} = [shape = circle, minimum size = 7pt, inner sep = 1pt] 
\node[vertex] (G--2) at (1.5, -1) [shape = circle, draw] {}; 
\node[vertex] (G--1) at (0.0, -1) [shape = circle, draw] {}; 
\node[vertex] (G-1) at (0.0, 1) [shape = circle, draw] {}; 
\node[vertex] (G-2) at (1.5, 1) [shape = circle, draw] {}; 
\draw[] (G-1) .. controls +(0.5, -0.5) and +(-0.5, -0.5) .. (G-2); 
\end{tikzpicture} 
- \frac{1}{n(n - 1)} \begin{tikzpicture}[scale = 0.5,thick, baseline={(0,-1ex/2)}] 
\tikzstyle{vertex} = [shape = circle, minimum size = 7pt, inner sep = 1pt] 
\node[vertex] (G--2) at (1.5, -1) [shape = circle, draw] {}; 
\node[vertex] (G--1) at (0.0, -1) [shape = circle, draw] {}; 
\node[vertex] (G-1) at (0.0, 1) [shape = circle, draw] {}; 
\node[vertex] (G-2) at (1.5, 1) [shape = circle, draw] {}; 
\draw[] (G--2) .. controls +(-0.5, 0.5) and +(0.5, 0.5) .. (G--1); 
\end{tikzpicture} 
 + \frac{1}{n - 1} \begin{tikzpicture}[scale = 0.5,thick, baseline={(0,-1ex/2)}] 
\tikzstyle{vertex} = [shape = circle, minimum size = 7pt, inner sep = 1pt] 
\node[vertex] (G--2) at (1.5, -1) [shape = circle, draw] {}; 
\node[vertex] (G--1) at (0.0, -1) [shape = circle, draw] {}; 
\node[vertex] (G-1) at (0.0, 1) [shape = circle, draw] {}; 
\node[vertex] (G-2) at (1.5, 1) [shape = circle, draw] {}; 
\draw[] (G--2) .. controls +(-0.5, 0.5) and +(0.5, 0.5) .. (G--1); 
\draw[] (G-1) .. controls +(0.5, -0.5) and +(-0.5, -0.5) .. (G-2); 
\end{tikzpicture} 
\end{multline*}
 and we may verify that the Young--Halverson--Ram matrix unit given above is idempotent, as expected. 
 We may also determine that 
\begin{multline*}
 e_{\big( \varnothing, \varnothing, {\Yvcentermath1 \Yboxdim{6.5pt} {\young(\null)}},
 \varnothing, {\Yvcentermath1 \Yboxdim{6.5pt} {\young(\null)}} 
 \big), \big( \varnothing, \varnothing, {\Yvcentermath1 \Yboxdim{6.5pt} {\young(\null)}},
 \varnothing, {\Yvcentermath1 \Yboxdim{6.5pt} {\young(\null)}} \big)} 
 = \\ -\frac{1}{n^{2}(n - 1)} \begin{tikzpicture}[scale = 0.5,thick, baseline={(0,-1ex/2)}] 
\tikzstyle{vertex} = [shape = circle, minimum size = 7pt, inner sep = 1pt] 
\node[vertex] (G--2) at (1.5, -1) [shape = circle, draw] {}; 
\node[vertex] (G--1) at (0.0, -1) [shape = circle, draw] {}; 
\node[vertex] (G-1) at (0.0, 1) [shape = circle, draw] {}; 
\node[vertex] (G-2) at (1.5, 1) [shape = circle, draw] {}; 
\end{tikzpicture} + 
 \frac{1}{n(n - 1)} \begin{tikzpicture}[scale = 0.5,thick, baseline={(0,-1ex/2)}] 
\tikzstyle{vertex} = [shape = circle, minimum size = 7pt, inner sep = 1pt] 
\node[vertex] (G--2) at (1.5, -1) [shape = circle, draw] {}; 
\node[vertex] (G--1) at (0.0, -1) [shape = circle, draw] {}; 
\node[vertex] (G-1) at (0.0, 1) [shape = circle, draw] {}; 
\node[vertex] (G-2) at (1.5, 1) [shape = circle, draw] {}; 
\draw[] (G-1) .. controls +(0.5, -0.5) and +(-0.5, -0.5) .. (G-2); 
\end{tikzpicture} + 
 \frac{1}{n(n - 1)} \begin{tikzpicture}[scale = 0.5,thick, baseline={(0,-1ex/2)}] 
\tikzstyle{vertex} = [shape = circle, minimum size = 7pt, inner sep = 1pt] 
\node[vertex] (G--2) at (1.5, -1) [shape = circle, draw] {}; 
\node[vertex] (G--1) at (0.0, -1) [shape = circle, draw] {}; 
\node[vertex] (G-1) at (0.0, 1) [shape = circle, draw] {}; 
\node[vertex] (G-2) at (1.5, 1) [shape = circle, draw] {}; 
\draw[] (G--2) .. controls +(-0.5, 0.5) and +(0.5, 0.5) .. (G--1); 
\end{tikzpicture} - 
 \frac{1}{n - 1} \begin{tikzpicture}[scale = 0.5,thick, baseline={(0,-1ex/2)}] 
\tikzstyle{vertex} = [shape = circle, minimum size = 7pt, inner sep = 1pt] 
\node[vertex] (G--2) at (1.5, -1) [shape = circle, draw] {}; 
\node[vertex] (G--1) at (0.0, -1) [shape = circle, draw] {}; 
\node[vertex] (G-1) at (0.0, 1) [shape = circle, draw] {}; 
\node[vertex] (G-2) at (1.5, 1) [shape = circle, draw] {}; 
\draw[] (G--2) .. controls +(-0.5, 0.5) and +(0.5, 0.5) .. (G--1); 
\draw[] (G-1) .. controls +(0.5, -0.5) and +(-0.5, -0.5) .. (G-2); 
\end{tikzpicture} + \\ \frac{n}{n - 1} \begin{tikzpicture}[scale = 0.5,thick, baseline={(0,-1ex/2)}] 
\tikzstyle{vertex} = [shape = circle, minimum size = 7pt, inner sep = 1pt] 
\node[vertex] (G--2) at (1.5, -1) [shape = circle, draw] {}; 
\node[vertex] (G--1) at (0.0, -1) [shape = circle, draw] {}; 
\node[vertex] (G-1) at (0.0, 1) [shape = circle, draw] {}; 
\node[vertex] (G-2) at (1.5, 1) [shape = circle, draw] {}; 
\draw[] (G-1) .. controls +(0.5, -0.5) and +(-0.5, -0.5) .. (G-2); 
\draw[] (G-2) .. controls +(0, -1) and +(0, 1) .. (G--2); 
\draw[] (G--2) .. controls +(-0.5, 0.5) and +(0.5, 0.5) .. (G--1); 
\draw[] (G--1) .. controls +(0, 1) and +(0, -1) .. (G-1); 
\end{tikzpicture} - 
 \frac{1}{n - 1} \begin{tikzpicture}[scale = 0.5,thick, baseline={(0,-1ex/2)}] 
\tikzstyle{vertex} = [shape = circle, minimum size = 7pt, inner sep = 1pt] 
\node[vertex] (G--2) at (1.5, -1) [shape = circle, draw] {}; 
\node[vertex] (G--1) at (0.0, -1) [shape = circle, draw] {}; 
\node[vertex] (G-2) at (1.5, 1) [shape = circle, draw] {}; 
\node[vertex] (G-1) at (0.0, 1) [shape = circle, draw] {}; 
\draw[] (G-2) .. controls +(0, -1) and +(0, 1) .. (G--2); 
\draw[] (G--2) .. controls +(-0.5, 0.5) and +(0.5, 0.5) .. (G--1); 
\draw[] (G--1) .. controls +(0.75, 1) and +(-0.75, -1) .. (G-2); 
\end{tikzpicture} - \frac{1}{n - 1} \begin{tikzpicture}[scale = 0.5,thick, baseline={(0,-1ex/2)}] 
\tikzstyle{vertex} = [shape = circle, minimum size = 7pt, inner sep = 1pt] 
\node[vertex] (G--2) at (1.5, -1) [shape = circle, draw] {}; 
\node[vertex] (G-1) at (0.0, 1) [shape = circle, draw] {}; 
\node[vertex] (G-2) at (1.5, 1) [shape = circle, draw] {}; 
\node[vertex] (G--1) at (0.0, -1) [shape = circle, draw] {}; 
\draw[] (G-1) .. controls +(0.5, -0.5) and +(-0.5, -0.5) .. (G-2); 
\draw[] (G-2) .. controls +(0, -1) and +(0, 1) .. (G--2); 
\draw[] (G--2) .. controls +(-0.75, 1) and +(0.75, -1) .. (G-1); 
\end{tikzpicture} + \frac{1}{n(n - 1)} \begin{tikzpicture}[scale = 0.5,thick, baseline={(0,-1ex/2)}] 
\tikzstyle{vertex} = [shape = circle, minimum size = 7pt, inner sep = 1pt] 
\node[vertex] (G--2) at (1.5, -1) [shape = circle, draw] {}; 
\node[vertex] (G-2) at (1.5, 1) [shape = circle, draw] {}; 
\node[vertex] (G--1) at (0.0, -1) [shape = circle, draw] {}; 
\node[vertex] (G-1) at (0.0, 1) [shape = circle, draw] {}; 
\draw[] (G-2) .. controls +(0, -1) and +(0, 1) .. (G--2); 
\end{tikzpicture}
\end{multline*}
 and we may again verify the desired idempotency. We may also apply the diagram basis expansions given above to verify that the 
 desired orthogonality relations such that 
\begin{multline*}
 e_{\big( \varnothing, \varnothing, {\Yvcentermath1 \Yboxdim{6.5pt} {\young(\null)}},
 \varnothing, \varnothing \big), \big( \varnothing, \varnothing, {\Yvcentermath1 \Yboxdim{6.5pt} {\young(\null)}},
 \varnothing, \varnothing \big)} 
 e_{\big( \varnothing, \varnothing, {\Yvcentermath1 \Yboxdim{6.5pt} {\young(\null)}},
 \varnothing, {\Yvcentermath1 \Yboxdim{6.5pt} {\young(\null)}} 
 \big), \big( \varnothing, \varnothing, {\Yvcentermath1 \Yboxdim{6.5pt} {\young(\null)}},
 \varnothing, {\Yvcentermath1 \Yboxdim{6.5pt} {\young(\null)}} \big)} 
 = \\
 e_{\big( \varnothing, \varnothing, {\Yvcentermath1 \Yboxdim{6.5pt} {\young(\null)}},
 \varnothing, {\Yvcentermath1 \Yboxdim{6.5pt} {\young(\null)}} 
 \big), \big( \varnothing, \varnothing, {\Yvcentermath1 \Yboxdim{6.5pt} {\young(\null)}},
 \varnothing, {\Yvcentermath1 \Yboxdim{6.5pt} {\young(\null)}} \big)} 
 e_{\big( \varnothing, \varnothing, {\Yvcentermath1 \Yboxdim{6.5pt} {\young(\null)}},
 \varnothing, \varnothing \big), \big( \varnothing, \varnothing, {\Yvcentermath1 \Yboxdim{6.5pt} {\young(\null)}},
 \varnothing, \varnothing \big)} = 0 
\end{multline*}
 hold. 
\end{example}

\section{Partition algebras as monoid algebras}
 According to Wilcox \cite{Wilcox2007} (cf.\ \cite{Clark1967}), for a semigroup $S$ and a commutative ring $R$ with unity, a 
 \emph{twisting} is a mapping $$ \alpha\colon S \times S \to R $$ such that $$ \alpha(x, y) \alpha(xy, z) = \alpha(x, yz) \alpha(y, z) $$ 
 for all $x, y, z \in S$, and 
 the \emph{twisted semigroup algebra} $R^{\alpha}[S]$ of $S$ over $R$ with a twisting $\alpha$ is the 
 $R$-algebra with $S$ as an $R$-basis and with a multiplicative operation $\cdot$ such that 
\begin{equation}\label{cdotalpha}
 x \cdot y = \alpha(x, y) (x, y) 
\end{equation}
 for $x, y \in S$, with \eqref{cdotalpha} extended linearly. Wilcox showed that $\mathbb{C}A_k(n)$ has the structure of a twisted 
 semigroup algebra, with the twisting being given by the power of $n$ in \eqref{pimultdefinition}. This leads us to consider how the 
 twisted semigroup algebra structure on $\mathbb{C}A_k(n)$ could be modified to provide a basis giving $\mathbb{C}A_k(n)$ a 
 semigroup algebra structure, or, better yet, a monoid algebra structure. It appears that this has not previously been considered, with 
 regard to extant literature related to diagram algebras and to Wilcox's work on twisted semigroup algebras 
 \cite{DolinkaEastEvangelouFitzGeraldHamHydeLoughlin2015,East2011,GuoXi2009,WangKoenig2008}. 

 For $n \in \mathbb{C} \setminus \{ 0, 1, \ldots, 2 k - 2 \}$, since $\mathbb{C}A_k(n)$ can be expressed as a direct sum of full matrix 
 algebras, this raises the following question. When does a finite direct sum of full matrix algebras admit the structure a group algebra, a 
 monoid algebra, or a semigroup algebra? For the case of semigroup algebras over algebraically closed fields, a necessary and sufficient 
 condition is given in the usual text on the algebraic theory of semigroups \cite[p.\ 167]{CliffordPreston1961}, via the following result 
 attributed to Hewitt and Zuckerman \cite{HewittZuckerman1955}. 

\begin{theorem}
 (Hewitt $\&$ Zuckerman, 1955) Let $A$ denote a semisimple algebra over a field $F$. If $A \cong FS$ for a finite semigroup $S$, then 
 one of the simple components of $A$ is of order $1$ over $F$. Moreover, the converse holds if $F$ is algebraically 
 closed \cite{HewittZuckerman1955}. 
\end{theorem}

 Since partition algebras have, to the best of our knowledge, not previously been considered 
 as (non-twisted) semigroup algebras, we highlight the following consequence of the 
 Hewitt--Zuckerman theorem. 

\begin{theorem}
 If $\mathbb{C}A_k(n)$ is semisimple, then $\mathbb{C}A_k(n)$ has the structure of a semigroup algebra. 
\end{theorem}

\begin{proof}
 This follows from the Hewitt--Zuckerman theorem, since there is at least one simple component of $\mathbb{C}A_k(n)$ of order $1$ over 
 the algebraically closed field $\mathbb{C}$, and this includes the $1$-dimensional component in the matrix algebra decomposition of 
 $\mathbb{C}A_k(n)$ corresponding to the vacillating tableau $ (\varnothing, \varnothing, (1), (1), \ldots, (k-1), (k-1), (k))$. 
\end{proof}

 To construct a basis of $\mathbb{C}A_k(n)$ that gives $\mathbb{C}A_k(n)$ the structure of a monoid algebra, we adapt an approach 
 we recent employed \cite{Campbellunpublished} in an unrelated context concerning a noncommutative Schur-like basis, and we invite 
 the interested reader to make the proper comparison. Informally, since $\mathbb{C}A_k(n)$ has two inequivalent $1$-dimensional 
 components in the decomposition of $\mathbb{C}A_k(n)$ as a direct sum of full matrix algebras, this can be thought of as allowing 
 for the application of a variant of the Hewitt--Zuckerman theorem in such a way so as to produce a monoid algebra, as opposed to a 
 semigroup algebra. 

\begin{theorem}\label{thmEbasis}
 If $\mathbb{C}A_k(n)$ is semisimple, then there exists a basis $M_{k} = M$ of $\mathbb{C}A_k(n)$ that has the structure of a monoid 
 under the product operation on $\mathbb{C}A_k(n)$. 
\end{theorem}

\begin{proof}
 We let $\{ e_{P, Q} \}_{P, Q}$ denote the basis of $\mathbb{C}A_k(n)$ consisting of Young--Halverson--Ram matrix units in $ 
 \mathbb{C}A_k(n)$, so that $P$ and $Q$ correspond to vacillating tableaux ending at level $k$ in $\hat{A}$, letting it be 
 understood that $e_{P, Q}$ is only defined if $P$ and $Q$ end with the same partition. For such vacillating tableaux $P$ and 
 $Q$, we set 
\[ \textsf{m}_{P, Q} 
 = \begin{cases} 
 e_{(\varnothing, \ldots, (1^k)), (\varnothing, \ldots, (1^k))} 
 & \text{if $P = Q$ and $P$ ends with $(1^{k})$,} \\ 
 \sum_{R} e_{R, R} & \text{if $P = Q$ and $P$ ends with $(k)$,} \\
 e_{(\varnothing, \ldots, (1^k)), (\varnothing, \ldots, (1^k))} 
 + e_{P, Q} & \text{otherwise,}
 \end{cases} 
\]
 letting it be understood that $\sum_{R} e_{R, R}$ denotes the sum over all idempotent elements in $\{ e_{P, 
 Q} \}_{P, Q}$, and letting it be understood that $(\varnothing, \ldots, (1^{k}))$ denotes the unique vacillating tableaux ending at level 
 $k$ and ending with $(1^{k})$. Let pairs of vacillating tableaux ending on level $k$ and ending with the same partition shape be 
 linearly ordered so that $((\varnothing$, $ \ldots$, $ (1^{k}))$, $ (\varnothing$, $ \ldots$, $ (1^{k})))$ is the first such pair followed 
 by $$ \big( (\varnothing, \varnothing, (1), (1), \ldots, (k-1), (k)), \ (\varnothing, \varnothing, (1), (1), \ldots, (k-1), (k)) \big). $$ 
 According to this ordering, by taking the transition matrix corresponding to the expansion of $\textsf{m}_{P, Q}$-expressions into the 
 Young--Halverson--Ram basis, this transition matrix has ones along the main diagonal and ones along the first column and certain 
 entries in the second row equal to one and zeroes everywhere else (cf.\ \cite{Campbellunpublished}). This transition matrix is 
 necessarily of determinant $1$, so that $M_k = M = \{ \textsf{m}_{P, Q} \}_{P, Q}$ 
 is a basis of $\mathbb{C}A_k(n)$. Writing $(\varnothing, \ldots, (k))$ in place of the unique 
 vacillating tableau ending at level $k$ and ending with $(k)$, we find that $$ \textsf{m}_{(\varnothing, \ldots, (k)), (\varnothing, \ldots, 
 (k))} \textsf{m}_{P, Q} = \textsf{m}_{P, Q} \textsf{m}_{(\varnothing, \ldots, (k)), (\varnothing, \ldots, (k))} = \textsf{m}_{P, Q}, $$ 
 since $ \textsf{m}_{(\varnothing, \ldots, (k)), (\varnothing, \ldots, (k))} $ is the sum of the idempotent matrix units in $\{ e_{P, 
 Q} \}_{P, Q}$. Now, let $P_{1}$, $P_{2}$, $P_{3}$, and $P_{4}$ denote vacillating tableaux ending at level $k$ such that the final entry 
 of $P_{i}$ is not $(k)$ for all $i \in \{ 1, 2, 3, 4 \}$. We may then use the matrix unit multiplication rules for the $e$-basis to 
 obtain that $$ \textsf{m}_{P_{1}, P_{2}} \textsf{m}_{P_{3}, P_4} = \textsf{m}_{ \left( \varnothing, \ldots, (1^{k}) \right), \left( \varnothing, 
 \ldots, (1^{k}) \right) } $$ if $P_{2} \neq P_{3}$, with $$ \textsf{m}_{P_{1}, P_{2}} \textsf{m}_{P_{3}, P_4} = \textsf{m}_{P_{1}, P_{4}} $$ 
 otherwise. We thus have that $M$ is a basis of $\mathbb{C}A_k(n)$ that is closed under the product operation on 
 $\mathbb{C}A_k(n)$, with the structure of a monoid under this operation, with the unique identity element being given by the 
 $P = Q = (\varnothing, \ldots, (k))$ case (and with associativity inherited from $\mathbb{C}A_k(n)$). 
\end{proof}

 This gives us a solution to the first part of the main research problem given in Section \ref{sectionIntro}. We now turn to the problem of 
 constructing an explicit, combinatorial rule, formulated in terms of partition diagrams, for multiplying elements of our new basis. 

\subsection{Monoid bases of partition algebras}
 Let it be understood that we are working with a semisimple partition algebra $\mathbb{C}A_k(n)$, and that the $n$-parameter involved 
 in the following definitions is the same as the $n$-parameter involved in the given partition algebra. For a given vacillating tableau $P 
 = (P_{1}, P_{2}, \ldots, P_{\ell(P)})$, we map this vacillating tableau to the tuple $$ \left( \left( n - | P_{i} | - \frac{1 + (-1)^{i}}{2} \right) 
 \cdot P_i : i \in \{ 1, 2, \ldots, \ell(P) \} \right), $$ where $\cdot$ denotes the operation of concatenation of tuples. Let $\text{HRtoBH}$ 
 denote this mapping. 

\begin{example}
 We find that $$ \text{{HRtoBH}}(\varnothing, \varnothing, \varnothing, \varnothing, \varnothing) = ((n), (n-1), (n), (n-1), (n)) $$ and 
 that $$ \text{{HRtoBH}}(\varnothing, \varnothing, {\Yvcentermath1 \Yboxdim{6.5pt} {\young(\null)}}, \varnothing, \varnothing) = ((n), 
 (n - 1), (n, 1), (n - 1), (n)). $$ 
\end{example}

 A \emph{Benkart--Halverson vacillating tableau} may then be defined as a tuple of the form $\text{{HRtoBH}}(P)$ for a vacillating 
 tableau $P$ in the sense of this term previously given and given by Halverson and Ram \cite{HalversonRam2005}. For the sake of 
 clarity, we may refer to these original vacillating tableaux as \emph{Halverson--Ram vacillating tableaux}. We also let $\text{{BHtoHR}}$ 
 denote the inverse of $\text{{HRtoBH}}$. 

 To be consistent with the notation and terminology conventions given by Benkart and Halverson \cite{BenkartHalverson2019trends}, we 
 adopt the following definition. We henceforth adopt the convention whereby of ordered elements are ordered according to the last 
 letter order, so that $s_{1} < s_{2}$ if $\text{max}(s_1) < \text{max}(s_2)$. 

\begin{definition}
 Let $\lambda \vdash n$ and write $\lambda^{\#} = \big( \lambda_{2}, \lambda_{3}, \ldots, \lambda_{\ell(\lambda)} \big)$. Let $t$ be an 
 integer such that $|\lambda^{\#}| \leq t \leq n$. A \emph{set-partition tableau $\mathcal{T}$ of shape $\lambda$ and content 
 $\{ 0^{n-t}, 1, 2 , \ldots, k \}$} is a labeling of the cells of $\lambda$ so that 

\begin{enumerate}

\item The leftmost $n - t$ cells in the first row of $\lambda$ are labeled with $\varnothing$; 

\item The cells in the skew shape $\lambda / (n-t)$ are labeled with distinct and pairwise disjoint subsets of $\{ 1, 2, \ldots, k \}$ in such a 
 way so that the union of the labels is $\{ 1, 2, \ldots, k \}$ and no cell inthe given skew shape is unlabeled; 

\item The rows and columns are strictly increasing. 

\end{enumerate}

\end{definition}

\begin{example}
 The set-partition tableaux of shape $\lambda \vdash 4$ and content $\{ 0^{4-t}$, $1$, $ 2 \}$ are as below, adopting the convention 
 whereby $\varnothing$-labeled cells may be denoted as unlabeled cells, and adopting the convention whereby set-valued labels 
 may be denoted without brackets or commas. 
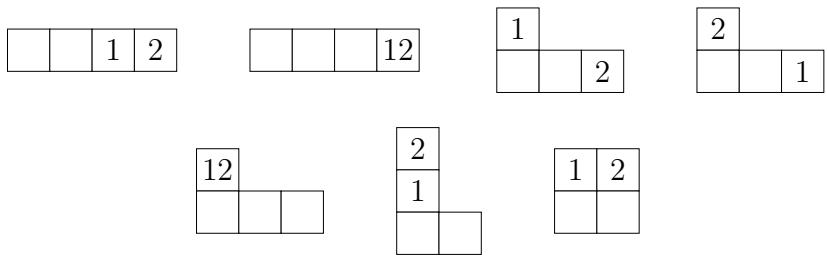
\begin{equation*}
 \Yvcentermath1 \Yboxdim{16pt} \young(\null{\null}12) \ \ \ \ \ \ \ 
 \Yvcentermath1 \Yboxdim{16pt} \young(\null{\null}\null\twelve) \ \ \ \ \ \ \ 
 \Yvcentermath1 \Yboxdim{16pt} \young(1,\null{\null}2) \ \ \ \ \ \ \ 
 \Yvcentermath1 \Yboxdim{16pt} \young(2,\null{\null}1) 
\end{equation*}
\begin{equation*}
 \Yvcentermath1 \Yboxdim{16pt} \young(\twelve,\null{\null}\null) \ \ \ \ \ \ \ 
 \Yvcentermath1 \Yboxdim{16pt} \young(2,1,{\null}\null) \ \ \ \ \ \ \ 
 \Yvcentermath1 \Yboxdim{16pt} \young(12,{\null}\null) 
\end{equation*}
\end{example}

 Benkart and Halverson \cite{BenkartHalverson2019trends} introduced a bijection $\text{BH}_{k} = \text{BH}$ from set-partition 
 tableaux of content $\{ 0^{\eta-t}, 1, 2 ,\ldots, k \}$ for fixed $\eta \geq 2k$ that we take as $\eta = 2k$ for convenience to 
 Benkart--Halverson vacillating tableaux of order $k$. This is to be required in our construction of a new basis for $\mathbb{C}A_k(n)$. 

 The RSK correspondence is often seen as a hallmark in the history of combinatorics. In its most familiar form, it provides a bijection 
 between the set of order-$n$ permutations and the set of pairs of standard Young tableaux of the same shape and with $n$ cells. 
 For the purposes of this paper, we require the below Theorem due (in an equivalent form) to Colmenarejo et al.\ 
 \cite{ColmenarejoOrellanaSaliolaSchillingZabrocki2020} and related to an extension of RSK correspondence for generalized 
 permutations. For background on the usual formulations of the RSK correspondence, we refer to the appropriate works by 
 Robinson \cite{Robinson1938}, Schensted \cite{Schensted1961}, and Knuth \cite{Knuth1970}. The interest in the construction of a 
 basis $M$ for $\mathbb{C}A_k(n)$ that gives $\mathbb{C}A_k(n)$ the structure of a monoid algebra with a product rule related, as 
 below, to combinatorial and representation-theoretic properties of set-partition tableaux is motivated by past applications of 
 set-partition tableaux within the field of combinatorial representation theory 
 \cite{BenkartHalverson2019Lond,BenkartHalverson2019trends,BenkartHalversonHarman2017,
Halverson2020,HalversonJacobson2020,OrellanaZabrocki2021}. 

\begin{theorem}\label{theoremColmenarejo}
 (Colmenarejo et al., 2020) For $n \geq 2k$, there is a bijection between the set of set-partitions of $\{ 1, 2, \ldots, k, 1', 2', \ldots, 
 k' \}$ and the set of pairs $(\mathcal{T}_{1}, \mathcal{T}_{2})$ of set-partition tableaux of the same shape $\lambda \vdash 
 n$ \cite{ColmenarejoOrellanaSaliolaSchillingZabrocki2020}. 
\end{theorem}

 As in the work of Benkart and Halverson \cite{BenkartHalverson2019trends}, we assume familiarity with Schensted row insertion 
 \cite[\S7.11]{Stanley1999}. We also assume familiarity with Knuth's formulation of the RSK algorithm for generalized permutations 
 \cite{Knuth1970}. In this direction, we require the following definition. 

\begin{definition}
 For two ordered alphabets $A$ and $B$, a \emph{generalized permutation} from $A$ to $B$ is a two-row array $ 
 \bigl(\begin{smallmatrix} a_{1} & a_{2} & \cdots & a_{r} \\ b_{1} & b_{2} & \cdots & b_{r} \end{smallmatrix} \bigr) $ such that $a_{1}$, 
 $a_{2}$, $\ldots$, $a_{r} \in A$ and $b_{1}$, $b_{2}$, $\ldots$, $b_{r} \in B$ and $a_{i} \leq_{A} a_{i+1}$ for $i \in \{ 1, 2, \ldots, r-1 \}$ 
 and $a_{i} = a_{i+1} \Longrightarrow b_{i} \leq_{B} b_{i+1}$. 
\end{definition}

 The tableaux on the right-hand side of \eqref{RSKRSK1RSK2} may be constructed according to the following modification of a 
 procedure given by Colmenarejo et al.\ \cite{ColmenarejoOrellanaSaliolaSchillingZabrocki2020}, and letting $\pi$ be written as a 
 set-partition of $\{ 1, 2, \ldots, k \} \cup \{ 1', 2', \ldots, k' \}$. For a propagating block $B$ of $\pi$, we write $B^{+} = B \cap \{ 1, 2, 
 \ldots, k \}$ and $B^{-} = B \cap \{ 1', 2', \ldots, k' \}$. Let $\{ \pi^{j_1}, \pi^{j_2}, \ldots, \pi^{j_{p}} \}$ denote the set of propagating 
 blocks of $\pi$ where the indices $j_{1}$, $j_2$, $\ldots$, $j_p$ are such that $$ \pi^{j_{1}, +} < \pi^{j_{2}, +} < \cdots < \pi^{j_{p}, 
 +} $$ in the last letter order, letting it be understood that $\pi^{j_{i}, +} = \big( \pi^{j_{p}} \big)^{+}$ and that $\pi^{j_{i}, -} = \big( 
 \pi^{j_{p}} \big)^{-}$. We then form a pair $(\mathcal{S}_{1}(\pi), \mathcal{S}_{2}(\pi))$ of tableaux via an application of the RSK 
 procedure to the generalized permutation $$ \left( \begin{matrix} \pi^{j_1, +} & \pi^{j_2, +} & \cdots & \pi^{j_p, +} \\ \pi^{j_1, 
 -} & \pi^{j_2, -} & \cdots & \pi^{j_p, -} \end{matrix} \right), $$ referring to the work of Knuth \cite{Knuth1970} and of 
 Colmenarejo et al.\ \cite{ColmenarejoOrellanaSaliolaSchillingZabrocki2020} for background and details. We then form the pair 
 $ (\text{RSK}_{1}(\pi), \text{RSK}_{2}(\pi))$ from $(\mathcal{S}_{1}(\pi)$, $ \mathcal{S}_{2}(\pi))$ as follows. 

 For a non-propagating block $B$ of $\pi$ that consists of vertices labeled with $\{ 1, 2, \ldots, k \}$ (resp.\ $\{ 1', 2', \ldots, k' \}$), we 
 take the set $\mathscr{S}_{1}(B)$ (resp.\ $\mathscr{S}_{2}(B)$) of labels of vertices in $B$. We then form the family $\{ 
 \mathscr{S}_{1}(B) \}_{B}$ (resp.\ $\{ \mathscr{S}_{2}(B) \}_{B}$) indexed by non-propagating blocks of the specified form, and we form 
 an increasing row $R_{1}$ (resp.\ $R_{2}$) consisting of $| \{ \mathscr{S}_{1}(B) \}_{B} |$ (resp.\ $| \{ \mathscr{S}_{2}(B) \}_{B} |$) cells 
 labeled with the elements in $ \{ \mathscr{S}_{1}(B) \}_{B}$ (resp.\ $\{ \mathscr{S}_{2}(B) \}_{B}$). We then add blank cells to the left of 
 $R_{1}$ (resp.\ $R_{2}$) to form a new row $(R_{1})'$ (resp.\ $(R_{2})'$) in such a way so that $\mathcal{S}_{1}(\pi)$ and $R_{1}$ 
 (resp.\ $\mathcal{S}_{2}(\pi)$ and $R_{1}$) together have $2k$ cells in total. We then form $\text{RSK}_{1}(\pi)$ (resp.\ 
 $\text{RSK}_{2}(\pi)$) as the tableau obtained by placing $\mathcal{S}_{1}(\pi)$ (resp.\ $\mathcal{S}_{2}(\pi)$) on top of $(R_{1})'$ 
 (resp.\ $(R_{2})'$). 

\begin{example}\label{examplebijCOSSZ}
 The bijection given Colmenarejo et al.\ \cite{ColmenarejoOrellanaSaliolaSchillingZabrocki2020} associated with the proof of Theorem 
 \ref{theoremColmenarejo} is illustrated below for the $k = 2$ case, again letting partition diagrams for $\mathbb{C}A_k(n)$ be 
 ordered according to the {\tt SageMath} ordering for partition diagrams. 
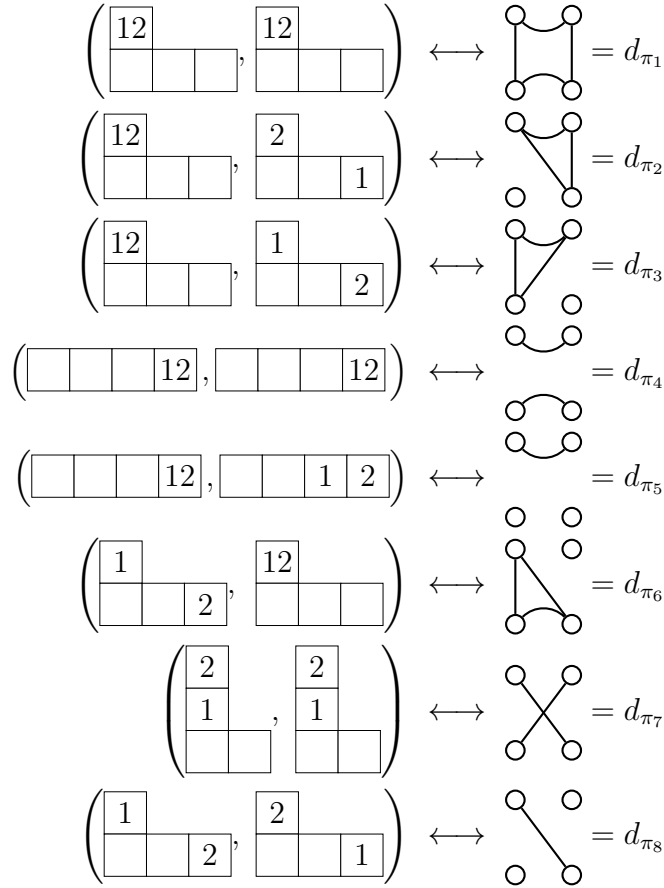
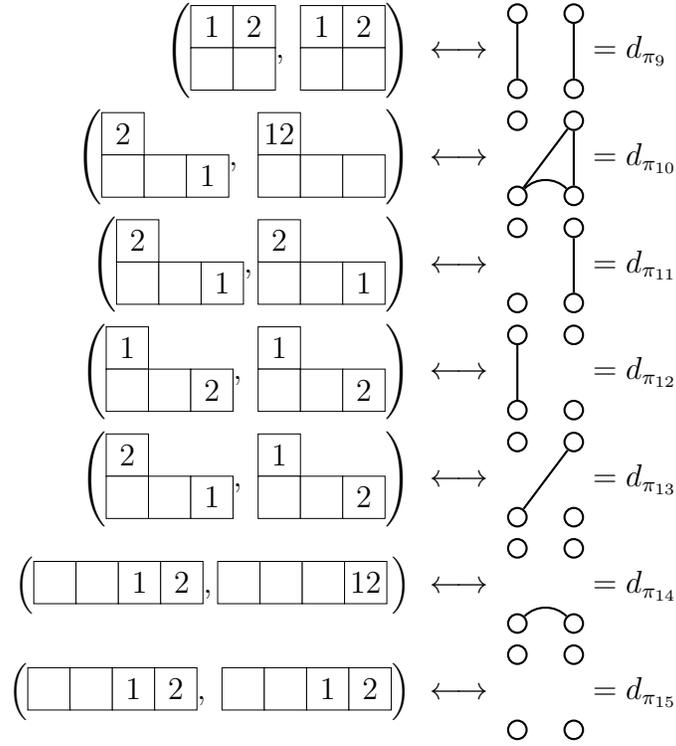
\begin{align*}
 \left( \Yvcentermath1 \Yboxdim{16pt} 
 \young(\twelve,\null{\null}\null), \Yvcentermath1 \Yboxdim{16pt} 
 \young(\twelve,\null{\null}\null) \right) & \ \longleftrightarrow \ 
 \begin{tikzpicture}[scale = 0.5,thick, baseline={(0,-1ex/2)}] 
\tikzstyle{vertex} = [shape = circle, minimum size = 7pt, inner sep = 1pt] 
\node[vertex] (G--2) at (1.5, -1) [shape = circle, draw] {}; 
\node[vertex] (G--1) at (0.0, -1) [shape = circle, draw] {}; 
\node[vertex] (G-1) at (0.0, 1) [shape = circle, draw] {}; 
\node[vertex] (G-2) at (1.5, 1) [shape = circle, draw] {}; 
\draw[] (G-1) .. controls +(0.5, -0.5) and +(-0.5, -0.5) .. (G-2); 
\draw[] (G-2) .. controls +(0, -1) and +(0, 1) .. (G--2); 
\draw[] (G--2) .. controls +(-0.5, 0.5) and +(0.5, 0.5) .. (G--1); 
\draw[] (G--1) .. controls +(0, 1) and +(0, -1) .. (G-1); 
\end{tikzpicture} = d_{\pi_{1}} \\ 
 \left( \Yvcentermath1 \Yboxdim{16pt} 
 \young(\twelve,\null{\null}\null), \ 
 \Yvcentermath1 \Yboxdim{16pt} 
 \young(2,\null{\null}1) 
 \right) & \ \longleftrightarrow \ 
 \begin{tikzpicture}[scale = 0.5,thick, baseline={(0,-1ex/2)}] 
 \tikzstyle{vertex} = [shape = circle, minimum size = 7pt, inner sep = 1pt] 
 \node[vertex] (G--2) at (1.5, -1) [shape = circle, draw] {}; 
 \node[vertex] (G-1) at (0.0, 1) [shape = circle, draw] {}; 
\node[vertex] (G-2) at (1.5, 1) [shape = circle, draw] {}; 
\node[vertex] (G--1) at (0.0, -1) [shape = circle, draw] {}; 
\draw[] (G-1) .. controls +(0.5, -0.5) and +(-0.5, -0.5) .. (G-2); 
\draw[] (G-2) .. controls +(0, -1) and +(0, 1) .. (G--2); 
\draw[] (G--2) .. controls +(-0.75, 1) and +(0.75, -1) .. (G-1); 
 \end{tikzpicture} = d_{\pi_{2}} \\ 
 \left( \Yvcentermath1 \Yboxdim{16pt} 
 \young(\twelve,\null{\null}\null), \ 
 \Yvcentermath1 \Yboxdim{16pt} 
 \young(1,\null{\null}2) 
 \right) & \ \longleftrightarrow \ 
\begin{tikzpicture}[scale = 0.5,thick, baseline={(0,-1ex/2)}] 
\tikzstyle{vertex} = [shape = circle, minimum size = 7pt, inner sep = 1pt] 
\node[vertex] (G--2) at (1.5, -1) [shape = circle, draw] {}; 
\node[vertex] (G--1) at (0.0, -1) [shape = circle, draw] {}; 
\node[vertex] (G-1) at (0.0, 1) [shape = circle, draw] {}; 
\node[vertex] (G-2) at (1.5, 1) [shape = circle, draw] {}; 
\draw[] (G-1) .. controls +(0.5, -0.5) and +(-0.5, -0.5) .. (G-2); 
\draw[] (G-2) .. controls +(-0.75, -1) and +(0.75, 1) .. (G--1); 
\draw[] (G--1) .. controls +(0, 1) and +(0, -1) .. (G-1); 
\end{tikzpicture} = d_{\pi_{3}} \\ 
 \left( \Yvcentermath1 \Yboxdim{16pt} 
 \young(\null{\null}\null\twelve), 
 \Yvcentermath1 \Yboxdim{16pt} 
 \young(\null{\null}\null\twelve) \right) 
 & \ \longleftrightarrow \ 
 \begin{tikzpicture}[scale = 0.5,thick, baseline={(0,-1ex/2)}] 
\tikzstyle{vertex} = [shape = circle, minimum size = 7pt, inner sep = 1pt] 
\node[vertex] (G--2) at (1.5, -1) [shape = circle, draw] {}; 
\node[vertex] (G--1) at (0.0, -1) [shape = circle, draw] {}; 
\node[vertex] (G-1) at (0.0, 1) [shape = circle, draw] {}; 
\node[vertex] (G-2) at (1.5, 1) [shape = circle, draw] {}; 
\draw[] (G--2) .. controls +(-0.5, 0.5) and +(0.5, 0.5) .. (G--1); 
\draw[] (G-1) .. controls +(0.5, -0.5) and +(-0.5, -0.5) .. (G-2); 
\end{tikzpicture} = d_{\pi_{4}} \\ 
 \left( \Yvcentermath1 \Yboxdim{16pt} 
 \young(\null{\null}\null\twelve), 
 \Yvcentermath1 \Yboxdim{16pt} 
 \young(\null{\null}12) \right) 
 & \ \longleftrightarrow \ 
 \begin{tikzpicture}[scale = 0.5,thick, baseline={(0,-1ex/2)}] 
\tikzstyle{vertex} = [shape = circle, minimum size = 7pt, inner sep = 1pt] 
\node[vertex] (G--2) at (1.5, -1) [shape = circle, draw] {}; 
\node[vertex] (G--1) at (0.0, -1) [shape = circle, draw] {}; 
\node[vertex] (G-1) at (0.0, 1) [shape = circle, draw] {}; 
\node[vertex] (G-2) at (1.5, 1) [shape = circle, draw] {}; 
\draw[] (G-1) .. controls +(0.5, -0.5) and +(-0.5, -0.5) .. (G-2); 
\end{tikzpicture} = d_{\pi_{5}} \\ 
 \left( \Yvcentermath1 \Yboxdim{16pt} 
 \young(1,\null{\null}2), \ 
 \Yvcentermath1 \Yboxdim{16pt} 
 \young(\twelve,\null{\null}\null) 
 \right) & \ \longleftrightarrow \ 
\begin{tikzpicture}[scale = 0.5,thick, baseline={(0,-1ex/2)}] 
\tikzstyle{vertex} = [shape = circle, minimum size = 7pt, inner sep = 1pt] 
\node[vertex] (G--2) at (1.5, -1) [shape = circle, draw] {}; 
\node[vertex] (G--1) at (0.0, -1) [shape = circle, draw] {}; 
\node[vertex] (G-1) at (0.0, 1) [shape = circle, draw] {}; 
\node[vertex] (G-2) at (1.5, 1) [shape = circle, draw] {}; 
\draw[] (G-1) .. controls +(0.75, -1) and +(-0.75, 1) .. (G--2); 
\draw[] (G--2) .. controls +(-0.5, 0.5) and +(0.5, 0.5) .. (G--1); 
\draw[] (G--1) .. controls +(0, 1) and +(0, -1) .. (G-1); 
\end{tikzpicture} = d_{\pi_{6}} \\ 
 \left( \Yvcentermath1 \Yboxdim{16pt} 
 \young(2,1,\null{\null}), \ 
 \Yvcentermath1 \Yboxdim{16pt} 
 \young(2,1,\null{\null}) 
 \right) & \ \longleftrightarrow \ 
\begin{tikzpicture}[scale = 0.5,thick, baseline={(0,-1ex/2)}] 
\tikzstyle{vertex} = [shape = circle, minimum size = 7pt, inner sep = 1pt] 
\node[vertex] (G--2) at (1.5, -1) [shape = circle, draw] {}; 
\node[vertex] (G-1) at (0.0, 1) [shape = circle, draw] {}; 
\node[vertex] (G--1) at (0.0, -1) [shape = circle, draw] {}; 
\node[vertex] (G-2) at (1.5, 1) [shape = circle, draw] {}; 
\draw[] (G-1) .. controls +(0.75, -1) and +(-0.75, 1) .. (G--2); 
\draw[] (G-2) .. controls +(-0.75, -1) and +(0.75, 1) .. (G--1); 
\end{tikzpicture} = d_{\pi_{7}} \\ 
 \left( \Yvcentermath1 \Yboxdim{16pt} 
 \young(1,\null{\null}2), \ 
 \Yvcentermath1 \Yboxdim{16pt} 
 \young(2,\null{\null}1) 
 \right) & \ \longleftrightarrow \ 
\begin{tikzpicture}[scale = 0.5,thick, baseline={(0,-1ex/2)}] 
\tikzstyle{vertex} = [shape = circle, minimum size = 7pt, inner sep = 1pt] 
\node[vertex] (G--2) at (1.5, -1) [shape = circle, draw] {}; 
\node[vertex] (G-1) at (0.0, 1) [shape = circle, draw] {}; 
\node[vertex] (G--1) at (0.0, -1) [shape = circle, draw] {}; 
\node[vertex] (G-2) at (1.5, 1) [shape = circle, draw] {}; 
\draw[] (G-1) .. controls +(0.75, -1) and +(-0.75, 1) .. (G--2); 
\end{tikzpicture} = d_{\pi_{8}} \\ 
 \left( \Yvcentermath1 \Yboxdim{16pt} 
 \young(12,\null{\null}), \ 
 \Yvcentermath1 \Yboxdim{16pt} 
 \young(12,\null{\null}) 
 \right) & \ \longleftrightarrow \ 
 \begin{tikzpicture}[scale = 0.5,thick, baseline={(0,-1ex/2)}] 
\tikzstyle{vertex} = [shape = circle, minimum size = 7pt, inner sep = 1pt] 
\node[vertex] (G--2) at (1.5, -1) [shape = circle, draw] {}; 
\node[vertex] (G-2) at (1.5, 1) [shape = circle, draw] {}; 
\node[vertex] (G--1) at (0.0, -1) [shape = circle, draw] {}; 
\node[vertex] (G-1) at (0.0, 1) [shape = circle, draw] {}; 
\draw[] (G-2) .. controls +(0, -1) and +(0, 1) .. (G--2); 
\draw[] (G-1) .. controls +(0, -1) and +(0, 1) .. (G--1); 
\end{tikzpicture} = d_{\pi_{9}} \\ 
 \left( \Yvcentermath1 \Yboxdim{16pt} 
 \young(2,\null{\null}1), \ 
 \Yvcentermath1 \Yboxdim{16pt} 
 \young(\twelve,\null{\null}\null) 
 \right) & \ \longleftrightarrow \ 
 \begin{tikzpicture}[scale = 0.5,thick, baseline={(0,-1ex/2)}] 
\tikzstyle{vertex} = [shape = circle, minimum size = 7pt, inner sep = 1pt] 
\node[vertex] (G--2) at (1.5, -1) [shape = circle, draw] {}; 
\node[vertex] (G--1) at (0.0, -1) [shape = circle, draw] {}; 
\node[vertex] (G-2) at (1.5, 1) [shape = circle, draw] {}; 
\node[vertex] (G-1) at (0.0, 1) [shape = circle, draw] {}; 
\draw[] (G-2) .. controls +(0, -1) and +(0, 1) .. (G--2); 
\draw[] (G--2) .. controls +(-0.5, 0.5) and +(0.5, 0.5) .. (G--1); 
\draw[] (G--1) .. controls +(0.75, 1) and +(-0.75, -1) .. (G-2); 
\end{tikzpicture} = d_{\pi_{10}} \\ 
 \left( \Yvcentermath1 \Yboxdim{16pt} 
 \young(2,\null{\null}1), \Yvcentermath1 \Yboxdim{16pt} 
 \young(2,\null{\null}1) \right) & \ \longleftrightarrow \ 
 \begin{tikzpicture}[scale = 0.5,thick, baseline={(0,-1ex/2)}] 
\tikzstyle{vertex} = [shape = circle, minimum size = 7pt, inner sep = 1pt] 
\node[vertex] (G--2) at (1.5, -1) [shape = circle, draw] {}; 
\node[vertex] (G-2) at (1.5, 1) [shape = circle, draw] {}; 
\node[vertex] (G--1) at (0.0, -1) [shape = circle, draw] {}; 
\node[vertex] (G-1) at (0.0, 1) [shape = circle, draw] {}; 
\draw[] (G-2) .. controls +(0, -1) and +(0, 1) .. (G--2); 
\end{tikzpicture} = d_{\pi_{11}} \\ 
 \left( \Yvcentermath1 \Yboxdim{16pt} 
 \young(1,\null{\null}2), \ 
 \Yvcentermath1 \Yboxdim{16pt} 
 \young(1,\null{\null}2) 
 \right) & \ \longleftrightarrow \ 
 \begin{tikzpicture}[scale = 0.5,thick, baseline={(0,-1ex/2)}] 
\tikzstyle{vertex} = [shape = circle, minimum size = 7pt, inner sep = 1pt] 
\node[vertex] (G--2) at (1.5, -1) [shape = circle, draw] {}; 
\node[vertex] (G--1) at (0.0, -1) [shape = circle, draw] {}; 
\node[vertex] (G-1) at (0.0, 1) [shape = circle, draw] {}; 
\node[vertex] (G-2) at (1.5, 1) [shape = circle, draw] {}; 
\draw[] (G-1) .. controls +(0, -1) and +(0, 1) .. (G--1); 
\end{tikzpicture} = d_{\pi_{12}} \\ 
 \left( \Yvcentermath1 \Yboxdim{16pt} 
 \young(2,\null{\null}1), \ 
 \Yvcentermath1 \Yboxdim{16pt} 
 \young(1,\null{\null}2) 
 \right) & \ \longleftrightarrow \ 
\begin{tikzpicture}[scale = 0.5,thick, baseline={(0,-1ex/2)}] 
\tikzstyle{vertex} = [shape = circle, minimum size = 7pt, inner sep = 1pt] 
\node[vertex] (G--2) at (1.5, -1) [shape = circle, draw] {}; 
\node[vertex] (G--1) at (0.0, -1) [shape = circle, draw] {}; 
\node[vertex] (G-2) at (1.5, 1) [shape = circle, draw] {}; 
\node[vertex] (G-1) at (0.0, 1) [shape = circle, draw] {}; 
\draw[] (G-2) .. controls +(-0.75, -1) and +(0.75, 1) .. (G--1); 
\end{tikzpicture} = d_{\pi_{13}} \\ 
 \left( \Yvcentermath1 \Yboxdim{16pt} 
 \young(\null{\null}12), 
 \Yvcentermath1 \Yboxdim{16pt} 
 \young(\null{\null}\null\twelve) \right) 
 & \ \longleftrightarrow \ 
 \begin{tikzpicture}[scale = 0.5,thick, baseline={(0,-1ex/2)}] 
\tikzstyle{vertex} = [shape = circle, minimum size = 7pt, inner sep = 1pt] 
\node[vertex] (G--2) at (1.5, -1) [shape = circle, draw] {}; 
\node[vertex] (G--1) at (0.0, -1) [shape = circle, draw] {}; 
\node[vertex] (G-1) at (0.0, 1) [shape = circle, draw] {}; 
\node[vertex] (G-2) at (1.5, 1) [shape = circle, draw] {}; 
\draw[] (G--2) .. controls +(-0.5, 0.5) and +(0.5, 0.5) .. (G--1); 
\end{tikzpicture} = d_{\pi_{14}} \\ 
 \left( \Yvcentermath1 \Yboxdim{16pt} 
 \young(\null{\null}12), \ 
 \Yvcentermath1 \Yboxdim{16pt} 
 \young(\null{\null}12) \right) 
 & \ \longleftrightarrow \ \begin{tikzpicture}[scale = 0.5,thick, baseline={(0,-1ex/2)}] 
 \tikzstyle{vertex} = [shape = circle, minimum size = 7pt, inner sep = 1pt] 
 \node[vertex] (G--2) at (1.5, -1) [shape = circle, draw] {}; 
 \node[vertex] (G--1) at (0.0, -1) [shape = circle, draw] {}; 
 \node[vertex] (G-1) at (0.0, 1) [shape = circle, draw] {}; 
 \node[vertex] (G-2) at (1.5, 1) [shape = circle, draw] {}; 
 \end{tikzpicture} = d_{\pi_{15}} 
\end{align*}
\end{example}

 We proceed to consider the application of the Benkart--Halverson bijection $\text{BH} = \text{BH}_{k}$ 
 \cite{BenkartHalverson2019trends}, which, as indicated above, for our purposes, is from the set of order-$k$ set-partition tableaux 
 with $\eta = 2k$ cells to the set of order-$k$ Benkart--Halverson tableaux. Borrowing notation from Benkart and Halverson, we write 
 $\textsf{T} \longleftarrow b$ to denote the operation of Schensted row insertion of a box $b$ (together with its label) into the 
 set-partition tableau $\textsf{T}$ (again subject to the specified ordering on sets of ordered sets). 

 Again with reference to the work of Benkart and Halverson \cite{BenkartHalverson2019trends}, we let $\textsf{T}$ be a set-partition 
 tableau of shape $\lambda \vdash 2k$ and content $[0^{\eta-t}, 1, \ldots, k]$, with $|\lambda^{\#}| \leq t \leq \eta$, and we then 
 apply the following procedure recursively to obtain an order-$k$ Benkart--Halverson vacillating tableau $$ \left( [\eta] = 
 \lambda^{(0)}, \lambda^{\left( \frac{1}{2} \right)}, \lambda^{(1)}, \ldots, \lambda^{(k)} = \lambda \right). $$ 

\begin{enumerate}

 \item Set $\lambda^{(k)} = \lambda$ and set $\textsf{T}^{(k)} = \textsf{T}$; 

\item For $j = k, k-1, \ldots, 1$, perform the following steps: 

\begin{enumerate}

 \item Set $\textsf{T}^{\left( j - \frac{1}{2} \right)}$ as the tableau that we obtain 
 from $\textsf{T}^{(j)}$ through the removal of the box $b$ containing $j$. 
 Set $\lambda^{\left( j -\frac{1}{2} \right)}$ as the shape of $\textsf{T}^{\left( j - \frac{1}{2} \right)}$; 

\item Remove the entry $j$ from $b$. If $b$ is empty afterwards, then add $0$ to it; and 

\item Set $\textsf{T}^{(j-1)} = \textsf{T}^{\left( j - \frac{1}{2} \right)} \longleftarrow b$, 
 and let $\lambda^{(j-1)}$ be the shape of $\textsf{T}^{(j-1)}$. 

\end{enumerate}

\end{enumerate}

\begin{remark} 
 The expression $\eta = 2k$ required in the above formulation of the Benkart--Halverson bijection is necessarily an integer, but, for 
 convenience, we may rewrite $\eta$ and $n$ to be consistent with our previous notation for Benkart--Halverson vacillating tableaux, 
 despite $n$ being complex. 
\end{remark}

\begin{example}
 The bijection $ \text{BH}_{2}$ is illustrated by the pairings given below, writing $n = 2k$ to be consistent with the notation 
 for Benkart--Halverson tableaux 
\begin{align*}
 \Yvcentermath1 \Yboxdim{16pt} 
 \young(\null{\null}12) & \ \longleftrightarrow \ 
 \big( (n), (n-1), (n-1, 1), (n-1), (n) \big) \\ 
 \Yvcentermath1 \Yboxdim{16pt} 
 \young(\null{\null}\null\twelve) & \ \longleftrightarrow \ 
 \big( (n), (n-1), (n), (n-1), (n) \big) \\ 
 \Yvcentermath1 \Yboxdim{16pt} 
 \young(1,\null{\null}2) & \ \longleftrightarrow \ 
 \big( (n), (n-1), (n-1, 1), (n-2, 1), (n-1, 1) \big) \\ 
 \Yvcentermath1 \Yboxdim{16pt} 
 \young(2,\null{\null}1) & \ \longleftrightarrow \ 
 \big( (n), (n-1), (n-1, 1), (n-1), (n-1, 1) \big) \\ 
 \Yvcentermath1 \Yboxdim{16pt} 
 \young(\twelve,\null{\null}\null) & \ \longleftrightarrow \ 
 \big( (n), (n-1), (n), (n-1), (n-1, 1) \big) \\ 
 \Yvcentermath1 \Yboxdim{16pt} 
 \young(2,1,{\null}\null) & \ \longleftrightarrow \ 
 \big( (n), (n-1), (n-1, 1), (n-2, 1), (n-2, 1, 1) \big) \\ 
 \Yvcentermath1 \Yboxdim{16pt} 
 \young(12,{\null}\null) & \ \longleftrightarrow \ 
 \big( (n), (n-1), (n-1, 1), (n-2, 1), (n-2, 2) \big) 
\end{align*}
\end{example}

 The {BHtoHR} bijection is given by mapping a Benkart--Halverson 
 vacillating tableau to the Halverson--Ram vacillating tableau $P$ obtained by removing 
 the initiat part of each entry in $P$. 

\begin{example}
 The bijection given by the composition 
 $\text{{BHtoHR}} \circ \text{BH}_{2}$ is illustrated by the pairings given below. 
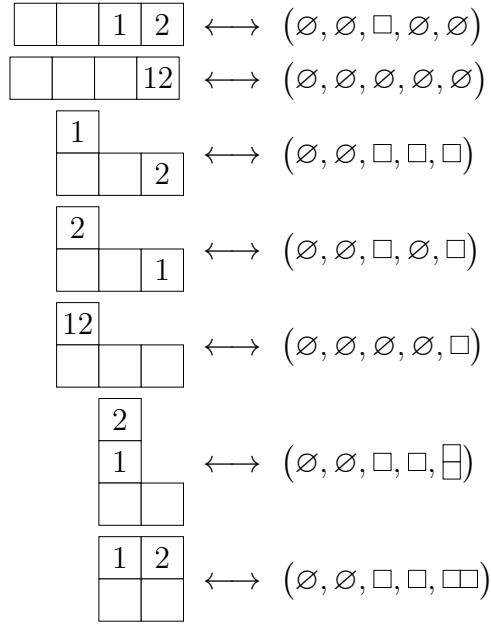
\begin{align*}
 \Yvcentermath1 \Yboxdim{16pt} 
 \young(\null{\null}12) & \ \longleftrightarrow \ 
 \big( \varnothing, \varnothing, 
 \Yvcentermath1 \Yboxdim{6.5pt} {\young(\null)}, 
 \varnothing, 
 \varnothing \big) \\ 
 \Yvcentermath1 \Yboxdim{16pt} 
 \young(\null{\null}\null\twelve) & \ \longleftrightarrow \ 
 \big( \varnothing, \varnothing, 
 \varnothing, 
 \varnothing, 
 \varnothing \big) \\ 
 \Yvcentermath1 \Yboxdim{16pt} 
 \young(1,\null{\null}2) & \ \longleftrightarrow \ 
 \big( \varnothing, \varnothing, \Yvcentermath1 \Yboxdim{6.5pt} {\young(\null)}, 
 \Yvcentermath1 \Yboxdim{6.5pt} {\young(\null)}, 
 \Yvcentermath1 \Yboxdim{6.5pt} {\young(\null)} \big) \\ 
 \Yvcentermath1 \Yboxdim{16pt} 
 \young(2,\null{\null}1) & \ \longleftrightarrow \ 
 \big( \varnothing, \varnothing, \Yvcentermath1 \Yboxdim{6.5pt} {\young(\null)}, 
 \varnothing, 
 \Yvcentermath1 \Yboxdim{6.5pt} {\young(\null)} \big) \\ 
 \Yvcentermath1 \Yboxdim{16pt} 
 \young(\twelve,\null{\null}\null) & \ \longleftrightarrow \ 
 \big( \varnothing, \varnothing, 
 \varnothing, 
 \varnothing, 
 \Yvcentermath1 \Yboxdim{6.5pt} {\young(\null)} \big) \\ 
 \Yvcentermath1 \Yboxdim{16pt} 
 \young(2,1,{\null}\null) & \ \longleftrightarrow \ 
 \big( \varnothing, \varnothing, 
 \Yvcentermath1 \Yboxdim{6.5pt} {\young(\null)}, \Yvcentermath1 \Yboxdim{6.5pt} {\young(\null)}, 
 \Yvcentermath1 \Yboxdim{6.5pt} {\young(\null,\null)} \big) \\ 
 \Yvcentermath1 \Yboxdim{16pt} 
 \young(12,{\null}\null) & \ \longleftrightarrow \ 
 \big( \varnothing, \varnothing, 
 \Yvcentermath1 \Yboxdim{6.5pt} {\young(\null)}, \Yvcentermath1 \Yboxdim{6.5pt} {\young(\null)}, 
 \Yvcentermath1 \Yboxdim{6.5pt} {\young(\null\null)} \big) 
\end{align*}
\end{example}

 Being consistent with our notation associated with Theorem \ref{theoremColmenarejo}, for an order-$k$ partition diagram $\pi$, 
 we write 
\begin{equation}\label{RSKRSK1RSK2}
 \text{RSK}(\pi) = \left( \text{RSK}_{1}(\pi), \, \text{RSK}_{2}(\pi) \right) 
\end{equation}
 to denote the pair of set-partition tableaux $ \text{RSK}_{1}(\pi)$ and $ \text{RSK}_{2}(\pi) $ of the same shape with $2k$ cells 
 obtained through the application to $\pi$ of the above RSK-like bijection for partition diagrams that is due to Colmenarejo et al.\ 
 \cite{ColmenarejoOrellanaSaliolaSchillingZabrocki2020}. We proceed to set 
\begin{equation*}
 m_{\pi} := \textsf{m}_{ \left( \text{{BHtoHR}}( \text{BH}_{k}(\text{RSK}_{1}(\pi))), 
 \text{{BHtoHR}}( \text{BH}_{k}(\text{RSK}_{2}(\pi)) ) \right) } 
\end{equation*}
 for a partition diagram of order $k$. 

\begin{theorem}
 If $\mathbb{C}A_k(n)$ is semisimple, then the family $\{ m_{\pi} \}_{\pi}$ consisting of expressions of the form $m_{\pi}$ for an 
 order-$k$ partition diagram $\pi$ is a basis of $\mathbb{C}A_k(n)$. 
\end{theorem}

\begin{proof}
 This follows from $\{ \textsf{m}_{P, Q} \}_{P, Q}$ being a basis of $\mathbb{C}A_k(n)$ for pairs $(P, Q)$ of vacillating tableaux ending 
 at level $k$ in $\hat{A}$ with the same final entry, together with the bijectivity of the RSK mapping for partition diagrams, the 
 bijectivity of Benkart and Halverson's mapping $\text{BH}_{k}$, and the bijectivity of the mapping $\text{{BHtoHR}}$. 
\end{proof}

\begin{remark}
 By mimicking our construction of $m_{\pi}$, with the use of matrix unit bases for $\mathbb{C}A_k(n)$ other than the 
 Young--Halverson--Ram matrix unit bases for $\mathbb{C}A_k(n)$, this can be used to produce infinite families of bases for 
 $\mathbb{C}A_k(n)$ that give $\mathbb{C}A_k(n)$ the structure of a monoid algebra (for the semisimple case). For convenience (due 
 to certain technical matters related to the scalar coefficients required in an explicit version of the relation in \eqref{HRrecsim} according 
 to Halverson and Ram's construction), we have used a variant $\{ \widetilde{e}^{\lambda}_{v_{1}, v_{2}} \}_{\lambda, v_1, v_2}$ of the 
 Young--Halverson--Ram matrix unit basis $\{ {e}^{\lambda}_{v_{1}, v_{2}} \}_{\lambda, v_1, v_2}$ of $\mathbb{C}A_2(n)$ to construct 
 the motivating example in Section \ref{subsectionmotivating}, with $\widetilde{e}^{\lambda}_{v_{1}, v_{2}} \sim {e}^{\lambda}_{v_{1}, 
 v_{2}}$ for every Young--Halverson--Ram matrix unit $ {e}^{\lambda}_{v_{1}, v_{2}} \in \mathbb{C}A_2(n)$ and with equality for the 
 idempotent case. 
\end{remark}

\begin{definition}
 We refer to the basis $M_{k} = M = \{ m_{\pi} \}_{\pi}$ as the \emph{monoid basis} of $\mathbb{C}A_k(n)$. 
\end{definition}

 We have thus constructed a basis $M_k$ of $\mathbb{C}A_k(n)$ that gives $\mathbb{C}A_k(n)$ the structure of a monoid algebra, again 
 for the non-degenerate case such that $n \in \mathbb{C} \setminus \{ 0, 1, \ldots, 2 k - 2 \}$, writing $\mathbb{C}A_k(n) 
 = \mathbb{C} M_k$. 

\subsection{Multiplication in the monoid basis}
 The goal of this section is to make the closure property in \eqref{malphambeta} explicit by formulating a combinatorial 
 rule for multiplying monoid basis elements. 

\begin{theorem}
 Let $\pi^{(1)}$ and $\pi^{(2)}$ be order-$k$ partition diagrams. For $i \in \{ 1, 2 \}$, if $\pi^{(i)}$ corresponds to the set-partition 
 $\{ \{ 1, 1' \}, \{ 2, 2' \}, \ldots, \{ k, k' \} \}$, then $m_{\pi^{(i)}}$ is the unique identity element in $M_{k}$. Suppose that neither 
 $\pi^{(1)}$ not $\pi^{(2)}$ is the identity element of $M_{k}$. If 
\begin{equation*}
 \text{\emph{RSK}}_{2}(\pi^{(1)}) \neq \text{\emph{RSK}}_{1}(\pi^{(2)}), 
\end{equation*}
 then 
\begin{equation}\label{indexedcross}
 m_{\pi^{(1)}} m_{\pi^{(2)}} = m_{\{ \{ 1, k' \}, \{ 2, (k-1)' \}, \ldots, \{ k, 1' \} \} }. 
\end{equation}
 Otherwise, we have that 
\begin{equation}\label{multthmOW}
 m_{\pi^{(1)}} m_{\pi^{(2)}} = m_{ \text{\emph{RSK}}^{-1}\left( \text{\emph{RSK}}_{1}(\pi^{(1)}), \ 
 \text{\emph{RSK}}_{2}(\pi^{(2)}) \right) }. 
\end{equation}
\end{theorem}

\begin{proof}
 This follows from the multiplication rules for the $\textsf{m}$-basis given in the proof of Theorem \ref{thmEbasis} together with the 
 property whereby the application of $\text{BH}_{k}^{-1} \circ \text{{HRtoBH}}$ to the order-$k$ Halverson--Ram vacillating tableau of 
 the form $$ (\varnothing, \varnothing, (1), (1), (1^2), (1^2), \ldots, (1^{k}) ) $$ produces the hook-shaped set-partition tableau 
 without any labels in the first row and with $k + 1$ cells in the first column, as may be verified inductively according to the algorithm 
 for $\text{BH}_{k}^{-1}$ given by Benkart and Halverson \cite{BenkartHalverson2019trends}. 
\end{proof}

\begin{example}
 We illustrate the product rules in both \eqref{indexedcross} and \eqref{multthmOW}, and we illustrate these product rules in contrast 
 to the diagram basis. Being consistent with the notation in Examples \ref{examplenotclosed}, \ref{exampleYHR}, and 
 \ref{examplebijCOSSZ}, we let the diagram basis element $d_{\pi}$ 
 be denoted with the partition diagram $\pi$ 
 with uncoloured vertices. By way of contrast, we let the monoid basis element 
 $m_{\pi}$ be denoted with a graph equivalent to $\pi$
 and with coloured vertices. 
 Being consistent with the ordering in Section \ref{subsectionmotivating} 
 for order-2 partition diagrams, we begin by considering the product 
\begin{equation*}
 d_{\pi_{2}} d_{\pi_{10}} = \left( \begin{tikzpicture}[scale = 0.5,thick, baseline={(0,-1ex/2)}] 
\tikzstyle{vertex} = [shape = circle, minimum size = 7pt, inner sep = 1pt] 
\node[vertex] (G--2) at (1.5, -1) [shape = circle, draw] {}; 
\node[vertex] (G-1) at (0.0, 1) [shape = circle, draw] {}; 
\node[vertex] (G-2) at (1.5, 1) [shape = circle, draw] {}; 
\node[vertex] (G--1) at (0.0, -1) [shape = circle, draw] {}; 
\draw[] (G-1) .. controls +(0.5, -0.5) and +(-0.5, -0.5) .. (G-2); 
\draw[] (G-2) .. controls +(0, -1) and +(0, 1) .. (G--2); 
\draw[] (G--2) .. controls +(-0.75, 1) and +(0.75, -1) .. (G-1); 
\end{tikzpicture} \right) \left( \begin{tikzpicture}[scale = 0.5,thick, baseline={(0,-1ex/2)}] 
\tikzstyle{vertex} = [shape = circle, minimum size = 7pt, inner sep = 1pt] 
\node[vertex] (G--2) at (1.5, -1) [shape = circle, draw] {}; 
\node[vertex] (G--1) at (0.0, -1) [shape = circle, draw] {}; 
\node[vertex] (G-2) at (1.5, 1) [shape = circle, draw] {}; 
\node[vertex] (G-1) at (0.0, 1) [shape = circle, draw] {}; 
\draw[] (G-2) .. controls +(0, -1) and +(0, 1) .. (G--2); 
\draw[] (G--2) .. controls +(-0.5, 0.5) and +(0.5, 0.5) .. (G--1); 
\draw[] (G--1) .. controls +(0.75, 1) and +(-0.75, -1) .. (G-2); 
\end{tikzpicture} \right) 
 = n \, d_{\pi_{1}} = n \, 
 \begin{tikzpicture}[scale = 0.5,thick, baseline={(0,-1ex/2)}] 
 \tikzstyle{vertex} = [shape = circle, minimum size = 7pt, inner sep = 1pt] 
 \node[vertex] (G--2) at (1.5, -1) [shape = circle, draw] {}; 
 \node[vertex] (G--1) at (0.0, -1) [shape = circle, draw] {}; 
 \node[vertex] (G-1) at (0.0, 1) [shape = circle, draw] {}; 
 \node[vertex] (G-2) at (1.5, 1) [shape = circle, draw] {}; 
 \draw[] (G-1) .. controls +(0.5, -0.5) and +(-0.5, -0.5) .. (G-2); 
 \draw[] (G-2) .. controls +(0, -1) and +(0, 1) .. (G--2); 
 \draw[] (G--2) .. controls +(-0.5, 0.5) and +(0.5, 0.5) .. (G--1); 
 \draw[] (G--1) .. controls +(0, 1) and +(0, -1) .. (G-1); 
 \end{tikzpicture}. 
\end{equation*}
 For the product 
\begin{equation*}
 m_{\pi_{2}} m_{\pi_{10}} = \left( \begin{tikzpicture}[scale = 0.5,thick, baseline={(0,-1ex/2)}] 
\tikzstyle{vertex} = [shape = circle, fill=blue, minimum size = 7pt, inner sep = 1pt] 
\node[vertex] (G--2) at (1.5, -1) [shape = circle, draw] {}; 
\node[vertex] (G-1) at (0.0, 1) [shape = circle, draw] {}; 
\node[vertex] (G-2) at (1.5, 1) [shape = circle, draw] {}; 
\node[vertex] (G--1) at (0.0, -1) [shape = circle, draw] {}; 
\draw[] (G-1) .. controls +(0.5, -0.5) and +(-0.5, -0.5) .. (G-2); 
\draw[] (G-2) .. controls +(0, -1) and +(0, 1) .. (G--2); 
\draw[] (G--2) .. controls +(-0.75, 1) and +(0.75, -1) .. (G-1); 
\end{tikzpicture} \right) \left( \begin{tikzpicture}[scale = 0.5,thick, baseline={(0,-1ex/2)}] 
\tikzstyle{vertex} = [shape = circle, fill=blue, minimum size = 7pt, inner sep = 1pt] 
\node[vertex] (G--2) at (1.5, -1) [shape = circle, draw] {}; 
\node[vertex] (G--1) at (0.0, -1) [shape = circle, draw] {}; 
\node[vertex] (G-2) at (1.5, 1) [shape = circle, draw] {}; 
\node[vertex] (G-1) at (0.0, 1) [shape = circle, draw] {}; 
\draw[] (G-2) .. controls +(0, -1) and +(0, 1) .. (G--2); 
\draw[] (G--2) .. controls +(-0.5, 0.5) and +(0.5, 0.5) .. (G--1); 
\draw[] (G--1) .. controls +(0.75, 1) and +(-0.75, -1) .. (G-2); 
\end{tikzpicture} \right), 
\end{equation*}
 the product rule in \eqref{multthmOW} gives us that 
\begin{align*}
 m_{\pi_{2}} m_{\pi_{10}} 
 & = m_{ \text{{RSK}}^{-1}\left( \Yvcentermath1 \Yboxdim{11pt} 
 \young(\twelve,\null{\null}\null), \ \Yvcentermath1 \Yboxdim{11pt} 
 \young(\twelve,\null{\null}\null) \right) } \\ 
 & = m_{\pi_1} \\ 
 & = \begin{tikzpicture}[scale = 0.5,thick, baseline={(0,-1ex/2)}] 
 \tikzstyle{vertex} = [shape = circle, fill=blue, minimum size = 7pt, inner sep = 1pt] 
 \node[vertex] (G--2) at (1.5, -1) [shape = circle, draw] {}; 
 \node[vertex] (G--1) at (0.0, -1) [shape = circle, draw] {}; 
 \node[vertex] (G-1) at (0.0, 1) [shape = circle, draw] {}; 
 \node[vertex] (G-2) at (1.5, 1) [shape = circle, draw] {}; 
 \draw[] (G-1) .. controls +(0.5, -0.5) and +(-0.5, -0.5) .. (G-2); 
 \draw[] (G-2) .. controls +(0, -1) and +(0, 1) .. (G--2); 
 \draw[] (G--2) .. controls +(-0.5, 0.5) and +(0.5, 0.5) .. (G--1); 
 \draw[] (G--1) .. controls +(0, 1) and +(0, -1) .. (G-1); 
 \end{tikzpicture}.
\end{align*}
 As for the product rule in \eqref{indexedcross}, we find 
 that $$ d_{\pi_7} d_{\pi_7} 
 = \left( 
 \begin{tikzpicture}[scale = 0.5,thick, baseline={(0,-1ex/2)}] 
\tikzstyle{vertex} = [shape = circle, minimum size = 7pt, inner sep = 1pt] 
\node[vertex] (G--2) at (1.5, -1) [shape = circle, draw] {}; 
\node[vertex] (G-1) at (0.0, 1) [shape = circle, draw] {}; 
\node[vertex] (G--1) at (0.0, -1) [shape = circle, draw] {}; 
\node[vertex] (G-2) at (1.5, 1) [shape = circle, draw] {}; 
\draw[] (G-1) .. controls +(0.75, -1) and +(-0.75, 1) .. (G--2); 
\draw[] (G-2) .. controls +(-0.75, -1) and +(0.75, 1) .. (G--1); 
\end{tikzpicture} \right) \left( 
 \begin{tikzpicture}[scale = 0.5,thick, baseline={(0,-1ex/2)}] 
\tikzstyle{vertex} = [shape = circle, minimum size = 7pt, inner sep = 1pt] 
\node[vertex] (G--2) at (1.5, -1) [shape = circle, draw] {}; 
\node[vertex] (G-1) at (0.0, 1) [shape = circle, draw] {}; 
\node[vertex] (G--1) at (0.0, -1) [shape = circle, draw] {}; 
\node[vertex] (G-2) at (1.5, 1) [shape = circle, draw] {}; 
\draw[] (G-1) .. controls +(0.75, -1) and +(-0.75, 1) .. (G--2); 
\draw[] (G-2) .. controls +(-0.75, -1) and +(0.75, 1) .. (G--1); 
\end{tikzpicture} \right) 
 = d_{\pi_9} 
 = \begin{tikzpicture}[scale = 0.5,thick, baseline={(0,-1ex/2)}] 
\tikzstyle{vertex} = [shape = circle, minimum size = 7pt, inner sep = 1pt] 
\node[vertex] (G--2) at (1.5, -1) [shape = circle, draw] {}; 
\node[vertex] (G-2) at (1.5, 1) [shape = circle, draw] {}; 
\node[vertex] (G--1) at (0.0, -1) [shape = circle, draw] {}; 
\node[vertex] (G-1) at (0.0, 1) [shape = circle, draw] {}; 
\draw[] (G-2) .. controls +(0, -1) and +(0, 1) .. (G--2); 
\draw[] (G-1) .. controls +(0, -1) and +(0, 1) .. (G--1); 
\end{tikzpicture}, $$
 whereas 
 that $$ m_{\pi_7} m_{\pi_7} 
 = \left( 
 \begin{tikzpicture}[scale = 0.5,thick, baseline={(0,-1ex/2)}] 
\tikzstyle{vertex} = [shape = circle, fill = blue, minimum size = 7pt, inner sep = 1pt] 
\node[vertex] (G--2) at (1.5, -1) [shape = circle, draw] {}; 
\node[vertex] (G-1) at (0.0, 1) [shape = circle, draw] {}; 
\node[vertex] (G--1) at (0.0, -1) [shape = circle, draw] {}; 
\node[vertex] (G-2) at (1.5, 1) [shape = circle, draw] {}; 
\draw[] (G-1) .. controls +(0.75, -1) and +(-0.75, 1) .. (G--2); 
\draw[] (G-2) .. controls +(-0.75, -1) and +(0.75, 1) .. (G--1); 
\end{tikzpicture} \right) \left( 
 \begin{tikzpicture}[scale = 0.5,thick, baseline={(0,-1ex/2)}] 
\tikzstyle{vertex} = [shape = circle, fill = blue, minimum size = 7pt, inner sep = 1pt] 
\node[vertex] (G--2) at (1.5, -1) [shape = circle, draw] {}; 
\node[vertex] (G-1) at (0.0, 1) [shape = circle, draw] {}; 
\node[vertex] (G--1) at (0.0, -1) [shape = circle, draw] {}; 
\node[vertex] (G-2) at (1.5, 1) [shape = circle, draw] {}; 
\draw[] (G-1) .. controls +(0.75, -1) and +(-0.75, 1) .. (G--2); 
\draw[] (G-2) .. controls +(-0.75, -1) and +(0.75, 1) .. (G--1); 
\end{tikzpicture} \right) 
 = m_{\pi_7} 
 = \begin{tikzpicture}[scale = 0.5,thick, baseline={(0,-1ex/2)}] 
\tikzstyle{vertex} = [shape = circle, fill = blue, minimum size = 7pt, inner sep = 1pt] 
\node[vertex] (G--2) at (1.5, -1) [shape = circle, draw] {}; 
\node[vertex] (G-1) at (0.0, 1) [shape = circle, draw] {}; 
\node[vertex] (G--1) at (0.0, -1) [shape = circle, draw] {}; 
\node[vertex] (G-2) at (1.5, 1) [shape = circle, draw] {}; 
\draw[] (G-1) .. controls +(0.75, -1) and +(-0.75, 1) .. (G--2); 
\draw[] (G-2) .. controls +(-0.75, -1) and +(0.75, 1) .. (G--1); 
\end{tikzpicture}. $$
\end{example}

\section{Conclusion}
 We conclude with the problem that is given below and that concerns the semisimplicity of partition algebras. 

 The Halverson--Ram recursive construction for matrix units for semisimple partition algebras is central to our construction. It would be 
 highly desirable, from computational, combinatorial, and representation-theoretic perspectives, to obtain a more explicit and 
 combinatorially meaningful matrix unit construction for partition algebras that provides a direct analogue of Young's (non-recursive) 
 matrix unit formula in \eqref{Youngmainformula}. This is an open problem, despite recent progress on this problem 
 \cite{Campbell2024,Campbell2025} and despite an alternative matrix unit construction for partition algebras introduced recently by 
 Padellaro \cite{Padellaro2023}. 

\subsection*{Acknowledgements}
 The author was supported through an NSERC Discovery Grant made available through Dalhousie University. The author is thankful to 
 Mike Zabrocki for very useful feedback related to this paper.

 \

John M.\ Campbell

\vspace{0.1in}

   Department of Mathematics and Statistics

  Dalhousie University

   Halifax, Nova Scotia,    Canada 

\vspace{0.1in}

{\tt jh241966@dal.ca}

\end{document}